\documentclass[english]{amsart}
\usepackage[T1]{fontenc}
\usepackage[latin1]{inputenc}
\usepackage[english]{babel}
\usepackage{amsmath,amsfonts,amssymb,amsthm,enumitem}
\usepackage{mathrsfs}
\usepackage{graphicx}

\usepackage{bbold}

\usepackage{xcolor}

\usepackage[pdftitle={Reduction of filtered K-theory},
  pdfauthor={Arklint, Bentmann, Katsura},
]{hyperref}
\usepackage[lite]{amsrefs}
\usepackage{microtype}

\numberwithin{equation}{section}

\theoremstyle{plain}
\newtheorem{theorem}[equation]{Theorem}
\newtheorem{proposition}[equation]{Proposition}
\newtheorem{lemma}[equation]{Lemma}
\newtheorem{corollary}[equation]{Corollary}

\theoremstyle{definition}
\newtheorem{definition}[equation]{Definition}

\theoremstyle{remark}
\newtheorem{remark}[equation]{Remark}

\renewcommand{\phi}{\varphi}
\renewcommand{\epsilon}{\varepsilon}
\newcommand{\N}{\mathbb N}
\newcommand{\Z}{\mathbb Z}

\newcommand{\csa}{$C^*$\nb-al\-ge\-bra}

\newcommand{\CK}{Cuntz--Krie\-ger algebra}

\renewcommand{\subset}{\subseteq} 

\DeclareMathOperator{\id}{id}

\newcommand*{\KK}{\textup{KK}}
\newcommand*{\K}{\textup{K}}

\DeclareMathOperator{\Sus}{\Sigma}

\newcommand{\I}[1]{\mathbb I_{#1}}

\DeclareMathOperator{\im}{im}

\DeclareMathOperator{\coker}{coker}

\newcommand{\Op}{\mathbb O}

\newcommand{\LC}{\mathbb{LC}}

\DeclareMathOperator{\FK}{FK}
\newcommand{\NT}{\mathcal{NT}}
\newcommand{\Mod}{\mathfrak{Mod}}

\newcommand{\onto}{\twoheadrightarrow}
\newcommand{\into}{\hookrightarrow}

\newcommand{\ob}{\widetilde\partial{}}

\newcommand{\rrzero}{real-rank-zero-like}

\DeclareMathOperator{\Nat}{Nat}
\newcommand{\Ab}{\mathfrak{Ab}}
\newcommand{\kk}{\mathfrak{KK}}

\DeclareMathOperator{\rank}{rank}
\newcommand{\unit}{\textnormal{unit}}
\newcommand{\pt}{\textnormal{pt}}

\newcommand{\Catgunnar}{\mathcal{R}}
\newcommand{\Cattakeshi}{\mathcal{B}}
\newcommand{\Catgunnartakeshi}{{\Cattakeshi\Catgunnar}}
\newcommand{\Catsixterm}{\mathcal{ST}}
\newcommand{\ot}{\leftarrow}

\newcommand{\Basis}{\mathbb{B}}
\newcommand{\Sheaf}{\mathfrak{Sh}}
\newcommand{\CoSheaf}{\mathfrak{CoSh}}

\DeclareMathOperator{\Path}{Path}
\DeclareMathOperator{\DoublePaths}{DP}

\newcommand{\FKtakeshi}{\FK_\Cattakeshi}
\newcommand{\FKgunnar}{\FK_\Catgunnar}
\newcommand{\Fgunnar}{\mathfrak{F}_\Catgunnar}
\newcommand{\Ftakeshi}{\mathfrak{F}_\Cattakeshi}
\newcommand{\Fsixterm}{\mathfrak{F}_\Catsixterm}
\newcommand{\Fgunnartakeshi}{\mathfrak{F}_\Catgunnartakeshi}
\newcommand{\FKgunnartakeshi}{\FK_\Catgunnartakeshi}
\newcommand{\FGTgunnar}{\mathfrak{F}_{\Catgunnartakeshi,\Catgunnar}}
\newcommand{\FKsixterm}{\FK_\Catsixterm}

\newcommand{\obd}[1]{\widetilde{\partial}\{{#1}\}}

\newcommand{\osi}[1]{\widetilde{\{{#1}\}}}
\newcommand{\csi}[1]{\overline{\{{#1}\}}}
\newcommand{\si}[1]{\{{#1}\}}

\newcommand{\gsi}[1]{{#1}_{1}}
\newcommand{\gobd}[1]{\widetilde\partial{#1}_{0}}
\newcommand{\gosi}[1]{\widetilde{{#1}}_0}
\newcommand{\tcsi}[1]{\overline{{#1}}_1}

\newcommand*{\Star}{\texorpdfstring{$^*$\nobreakdash-}{*-}}
\newcommand{\shom}{$^*$\nobreakdash-homo\-mor\-phism}
\newcommand*{\nb}{\nobreakdash}

\input xy
\xyoption{all}
\newdir{ >}{{}*!/-5pt/@{>}}

\title[Reduction of filtered $\K$\nb-the\-ory]{Reduction of filtered $\K$\nb-the\-ory and\\ a characterization of Cuntz--Krie\-ger algebras}

\author{Sara E. Arklint}
\address{Department of Mathematical Sciences, University of Copenhagen, Uni\-versi\-tets\-park\-en~5, DK-2100 Copenhagen, Denmark}
\email{arklint@math.ku.dk}

\author{Rasmus Bentmann}
\address{Department of Mathematical Sciences, University of Copenhagen, Uni\-versi\-tets\-park\-en~5, DK-2100 Copenhagen, Denmark}
\email{bentmann@math.ku.dk}

\author{Takeshi Katsura}
\address{Department of Mathematics, Keio University, 3-14-1 Hiyoshi, Kouhoku-ku, Yokohama 223-8522, Japan}
\email{katsura@math.keio.ac.jp }

\thanks{This research was supported by the Danish National Research Foundation through the Centre for Symmetry and Deformation (DNRF92). The third-named author was partially supported by the Japan Society for the Promotion of Science.}

\begin{document}
\bibliographystyle{alpha}

\begin{abstract}
We show that filtered $\K$\nb-the\-ory is equivalent to a substantially smaller invariant for all real-rank-zero \csa s with certain primitive ideal spaces -- including the infinitely many so-called accordion spaces for which filtered $\K$\nb-the\-ory is known to be a complete invariant. As a consequence, we give a characterization of purely infinite Cuntz--Krie\-ger algebras whose primitive ideal space is an accordion space.
\end{abstract}

\maketitle

\section{Introduction}

The Cuntz and \CK s are historically and in general of great importance for our understanding of simple and non-simple purely infinite \csa s as they were not only the first constructed examples of such but are also very tangible due to the combinatorial nature of their construction~\cite{cuntz_krieger}.
\CK s arise from shifts of finite type and it has been shown that they are exactly the graph \csa{}s $C^*(E)$ arising from finite directed graphs $E$ with no sources~\cite{arklint_ruiz}.
Using the Kirchberg--Phillips classification theorem~\cites{Kirchberg:Fields_Monograph,Phillips:Classification}, the Cuntz algebras and simple \CK s can be identified, up to isomorphism, as the unital UCT Kirchberg algebras with a specific type of $\K$\nb-the\-ory \cites{Cuntz:O_n,Rordam:Class_of_CK_algs}. A~similar characterization for non-simple, purely infinite \CK s and, more generally, of unital graph \csa{}s of this type is desirable.

A Kirchberg $X$\nb-al\-ge\-bra is a purely infinite, nuclear, separable \csa{} with primitive ideal space homeomorphic to~$X$ (in a specified way).
When $X$ is a so-called {accordion space}, see~Definition~\ref{def:accordion}, the invariant \emph{filtered $\K$\nb-the\-ory} $\FK$ is a strongly complete invariant for stable Kirchberg $X$\nb-al\-ge\-bras with simple subquotients in the bootstrap class \cites{kirchberg,MN:Filtrated,bentmann_koehler}.  In particular, filtered $\K$\nb-the\-ory is complete for purely infinite graph \csa{}s with primitive ideal space of accordion type, and the main goal of this paper is to use this to achieve a \emph{characterization} in the sense of the previous paragraph of such purely infinite \CK s and graph \csa{}s. Since a \CK{} is purely infinite if and only if it has real rank zero (and more generally, a purely infinite graph \csa{} always has real rank zero~\cite{hongszymanski}), we will specifically investigate filtered $\K$\nb-the\-ory for \csa s of real rank zero.

In the companion paper~\cite{range_result}, we determine the range of \emph{reduced} filtered $\K$\nb-the\-ory with respect to purely infinite \CK s and graph \csa{}s. This invariant was originally defined by Gunnar Restorff~\cite{restorff}, who used it to give an ``internal'' classification of purely infinite \CK s, inspired by work of Mikael R\o rdam~\cite{rordam} and work of Mike Boyle and Danrun Huang on dynamical systems~\cite{boyle_huang}. 
In the present note, we show that under some assumptions on the primitive ideal space---which are satisfied for accordion spaces---the invariants filtered $\K$\nb-the\-ory and reduced filtered $\K$\nb-the\-ory are in a certain sense equivalent when restricted to purely infinite graph \csa{}s.

To be more precise, we show that isomorphisms on the reduced filtered $\K$\nb-the\-ory of purely infinite graph \csa{}s over so-called EBP spaces lift to isomorphisms on \emph{concrete} filtered $\K$\nb-the\-ory -- this invariant may be considered as a more explicit model of filtered $\K$\nb-the\-ory: the two are known to coincide for many spaces but not in general (compare Remark~\ref{rem:Q}).
Along the way, we introduce filtered $\K$\nb-the\-ory \emph{restricted to the canonical base}, denoted~$\FKtakeshi$, and show that, for real-rank-zero \csa s over an EBP space, isomorphisms on~$\FKtakeshi$ lift to isomorphisms on concrete filtered $\K$\nb-the\-ory.

For accordion spaces, our results furnish one-to-one correspondences, induced by the different variants of filtered $\K$\nb-the\-ory, between purely infinite graph \csa{}s respectively unital purely infinite graph \csa{}s or purely infinite \CK s on the one hand, and certain types of modules in the respective target categories on the other hand.
In particular, we obtain the desired characterization of purely infinite \CK s with accordion spaces as primitive ideal spaces:
\begin{theorem} \label{thm:nophantoms}
Let $A$ be a \csa{} whose primitive ideal space is an accordion space.
Then $A$ is a purely infinite \CK{} if and only if $A$ satisfies the following:
\begin{itemize}
\item $A$ is unital, purely infinite, nuclear, separable, and of real rank zero,
\item for all ideals $I$ and $J$ of $A$ with $I\subseteq J$ and $J/I$ simple, the quotient~$J/I$ belongs to the bootstrap class, the group~$\K_*(J/I)$ is finitely generated, the group~$\K_1(J/I)$ is free and $\rank\K_1(J/I)=\rank\K_0(J/I)$.
\end{itemize}
\end{theorem}
In the terms introduced by the first named author in~\cite{arklint}, our Theorem~\ref{thm:nophantoms} states that there is no \emph{phantom \CK{}} whose primitive ideal space is an accordion space.
It is an open question whether this holds for all finite primitive ideal spaces.

\subsection{Historical account}

By a seminal result of Eberhard Kirchberg, $\KK(X)$-equivalences between stable Kirchberg $X$\nb-al\-ge\-bras, that is, stable, tight, $\mathcal O_\infty$-ab\-sorb\-ing, nuclear, separable \csa s over a space $X$, lift to $X$\nb-equi\-vari\-ant \Star{}iso\-mor\-phisms. In~\cite{MN:Filtrated}, Ralf Meyer and Ryszard Nest established a Universal Coefficient Theorem computing the equivariant bivariant theory $\KK(X)$ from filtered $\K$\nb-the\-ory under the assumption that the topology of~$X$ is finite and totally ordered.
As a result, for such spaces~$X$, isomorphisms on filtered $\K$\nb-the\-ory between stable Kirchberg $X$\nb-al\-ge\-bras with simple subquotients in the bootstrap class lift to $X$\nb-equi\-vari\-ant \Star isomorphisms. This result was generalized in~\cites{bentmann_koehler} by the second-named author and Manuel K\"ohler to the case of so-called accordion spaces. 
Building on these results, S\o{}ren Eilers, Gunnar Restorff, and Efren Ruiz classified in~\cite{err} certain classes of real-rank-zero (not necessarily purely infinite) graph \csa{}s using \emph{ordered} filtered $\K$\nb-the\-ory.

On the other hand, Meyer--Nest and the second-named author have constructed counterexamples to the analogous classification statement over all six four-point non-accordion connected $T_0$\nb-spaces. More precisely, for each of these spaces~$X$, they exhibit two non-$\KK(X)$-equivalent Kirchberg $X$\nb-al\-ge\-bras with simple subquotients in the bootstrap class whose filtered $\K$\nb-the\-ory is isomorphic (see~\cites{MN:Filtrated,bentmann}).

Despite this obstruction, it had previously been shown by Gunnar Restorff in~\cite{restorff} that filtered $\K$\nb-the\-ory---in fact reduced filtered $\K$\nb-the\-ory---is a complete invariant for purely infinite Cuntz--Krie\-ger algebras. Any finite $T_0$-space, in particular the six problematic four-point spaces mentioned above, can be realized as the primitive ideal space of a purely infinite Cuntz--Krie\-ger algebra. Unfortunately, Restorff's result only gives an \emph{internal} classification of Cuntz--Krie\-ger algebras and admits no conclusion concerning when a given Cuntz--Krie\-ger algebra is stably isomorphic to a given purely infinite, nuclear, separable \csa{} with the same ideal structure and filtered $\K$\nb-the\-ory.

In~\cite{arr}, Gunnar Restorff, Efren Ruiz, and the first-named author noted that, for five of the six problematic four-point spaces, the constructed counterexamples to classification do \emph{not} have real rank zero. They went on to show that for four of these spaces $X$, filtered $\K$\nb-the\-ory is in fact a complete invariant for Kirchberg $X$\nb-al\-ge\-bras of real rank zero with simple subquotients in the bootstrap class. The four-point non-accordion space for which the constructed counterexample does have real rank zero will be denoted by~$\mathcal D$.

It is a general property of Cuntz--Krie\-ger algebras that the $\K_1$-group of every subquotient is free. The same is true, more generally, for graph \csa{}s. We observe that, for real-rank-zero \csa s over~$\mathcal D$ satisfying this $\K$\nb-the\-o\-retic condition, isomorphisms on the reduced filtered $\K$\nb-the\-ory lift to $\KK(\mathcal D)$-equivalences (see Proposition~\ref{prop:diamond}). There are therefore no known counterexamples to classification by filtered $\K$\nb-the\-ory of Kirchberg $X$\nb-al\-ge\-bras with simple subquotients in the bootstrap class that have the $\K$\nb-the\-ory of a real-rank-zero graph \csa{}.

\subsection{Organization of the paper}

After fixing some basic conventions and definitions in Section~\ref{sec:notation}, we introduce filtered $\K$\nb-the\-ory $\FK$ and \emph{concrete} filtered $\K$\nb-the\-ory $\FKsixterm$ in Section~\ref{sec:filtktheory}. Section~\ref{sec:sheaves} contains some basic definitions and facts concerning sheaves and cosheaves.

In Section~\ref{sec:takeshi_invariant}, filtered $\K$\nb-the\-ory restricted to the canonical base $\FKtakeshi$ is defined for spaces with the unique path property. We introduce EBP spaces and show that the concrete filtered $\K$\nb-the\-ory $\FKsixterm(A)$ of a real-rank-zero \csa{}~$A$ over an EBP space is completely determined by the filtered $\K$\nb-the\-ory restricted to the canonical base $\FKtakeshi(A)$, see Corollary~\ref{cor:sixterm_to_takeshi}.

In Section~\ref{gunnar}, reduced filtered $\K$\nb-the\-ory $\FKgunnar$ is defined, and it is shown in Section~\ref{sec:intermediate} that the concrete filtered $\K$\nb-the\-ory $\FKsixterm(A)$ of a real-rank-zero \csa{}~$A$ over an EBP space satisfying that all subquotients have free $\K_1$-groups can be recovered from the reduced filtered $\K$\nb-the\-ory $\FKgunnar(A)$, see~Corollary~\ref{cor:gunnartosixterm_cstar}. This is of particular interest because of the range results from~\cite{range_result} for (unital) reduced filtered $\K$\nb-the\-ory on (unital) purely infinite graph \csa{}s, see~Theorem~\ref{thm:range} (and~\ref{thm:range_unit}). In order to proceed from reduced to concrete filtered $\K$\nb-the\-ory in Section~\ref{sec:intermediate}, an ``intermediate'' invariant is introducted, which serves only technical purposes.

In Sections~\ref{sec:unital} and~\ref{sec:order}, unital filtered $\K$\nb-the\-ory and ordered filtered $\K$\nb-the\-ory are treated. The most complete results in our framework are possible for \csa{}s with primitive ideal spaces of accordion type; these are summarized in Section~\ref{sec:accordion}.
 
\subsection{Acknowledgements}
Most of this work was done while the third-named author stayed 
at the University of Copenhagen. 
He would like to thank the people in Copenhagen for their hospitality.
The authors are grateful to S{\o{}}ren Eilers for his encouragement and valuable comments.
We thank Mikael R{\o{}}rdam for helpful comments.
The second-named author thanks Ralf Meyer for the supervision of~\cite{bentmann}
which has influenced parts of this work.

\section{Notation} \label{sec:notation}
In this article, matrices act from the right and the composite of maps $A\xrightarrow{f} B\xrightarrow{g} C$ is denoted by~$fg$.
The category of abelian groups is denoted by $\Ab$, the category of $\Z/2$-graded abelian groups by $\Ab^{\mathbb Z/2}$.

Let $X$ be a finite $T_0$-space.
For a subset $Y$ of $X$, 
we let $\overline Y$ denote the closure of $Y$ in $X$, and 
let $\overline\partial Y$ denote the boundary 
$\overline Y \setminus Y$ of $Y$. 
Since $X$ is a finite space, 
there exists a smallest open subset $\widetilde Y$ of $X$ containing $Y$.
We let $\ob Y$ denote the set $\widetilde Y\setminus Y$.
For $x,y\in X$ we write $x\leq y$ 
when $\overline{\{x\}} \subset \overline{\{y\}}$, 
and $x < y$ when $x\leq y$ and $x \neq y$. 
We write $y\to x$ when $x < y$ and 
no $z \in X$ satisfies $x < z < y$.
A \emph{path} from $y$ to $x$ is a sequence $(z_k)_{k=1}^n$ such that $z_{k+1}\to z_{k}$ for $k=1, \ldots, n-1$ 
and $z_1=x$, $z_{n}=y$. We let $\Path(y,x)$ denote the set of paths from $y$ to $x$.

\begin{definition} \label{def:accordion}
An \emph{accordion space} is a $T_0$-space $X=\{x_1,\ldots,x_n\}$ 
such that for every $k=1,2,\ldots,n-1$  either $x_k \to x_{k+1}$
or $x_k \leftarrow x_{k+1}$ holds and such that $x_k \to x_l$
does not hold for any $k,l$ with $|k-l|\neq 1$.
\end{definition}

For instance, if $X$ is linear, that is, if $X=\{x_1,\ldots,x_n\}$ with
$x_n\to\cdots\to x_2\to x_1$, then $X$ is an accordion space.

\section{Filtered \texorpdfstring{$\K$}{K}-theory} \label{sec:filtktheory}
In this section filtered $\K$\nb-the\-ory and concrete filtered $\K$\nb-the\-ory are defined. Some properties of objects in their target categories are introduced.

A \emph{\csa{} $A$ over~$X$} is (equivalently given by) a \csa{} $A$ equipped with an infima- and suprema-preserving map $\Op(X)\to\I{}(A), U\mapsto A(U)$ mapping open subsets in $X$ to (closed, two-sided) ideals in~$A$ (in particular it holds that $A(\emptyset)=0$ and $A(X)=A$).
The \csa{} $A$ is called \emph{tight} over~$X$ if the map is a lattice-isomorphism.
A \shom{} $\phi\colon A\to B$ for \csa s $A$ and $B$ over~$X$ is called \emph{$X$\nb-equi\-vari\-ant} if $\phi\bigl(A(U)\bigr)\subseteq B(U)$ for all $U\in\Op(X)$.
Let $\LC(X)$ denote the set of locally closed subsets of $X$, that is, subsets of the form $U\setminus V$ with $U$ and $V$ open subsets of $X$ satisfying $V\subseteq U$.
For $Y\in\LC(X)$, and $U,V\in\Op(X)$ satisfying that $Y=U\setminus V$ and $U\supseteq V$, we define $A(Y)$ as the subquotient $A(Y)=A(U)/A(V)$, which up to natural isomorphism is independent of the choice of $U$ and $V$ (see \cite{MN:Bootstrap}*{Lemma 2.15}).

\begin{definition}
A tight, $\mathcal O_\infty$-absorbing, nuclear, separable \csa{} over~$X$ is called a \emph{Kirchberg $X$\nb-al\-ge\-bra}.
\end{definition}

Let $\kk(X)$ be the additive category 
whose objects are separable \csa s over~$X$ 
and whose set of morphisms from $A$ to $B$ is the Kasparov group 
$\KK_0(X;A,B)$ defined by Kirchberg 
(see \cite{MN:Bootstrap}*{Section 3} for details).
For a \csa{} $A$ over~$X$, 
a $\Z/2$-graded abelian group $\FK_Y^*(A)$ 
is defined as $\K_*\bigl(A(Y)\bigr)$ for all $Y\in\LC(X)$.  
Thus $\FK_Y^*$ is an additive funtor from $\kk(X)$ 
to the category $\Ab^{\mathbb Z/2}$ of $\Z/2$-graded abelian groups. 
Ralf Meyer and Ryszard Nest constructed in~\cite{MN:Filtrated} \csa s $R_Y$ over~$X$ satisfying that the functors $\FK_Y^*$ and $\KK_*(X;R_Y,-)$ are naturally isomorphic.

In their definition of filtered $\K$\nb-the\-ory $\FK^*$, Meyer--Nest consider the $\Z/2$-graded pre-additive category $\NT_*$ with objects $\LC(X)$ and morphisms
\[
\Nat_*(\FK_Y^*,\FK_Z^*)\cong\KK_*(X;R_Z,R_Y)
\]
between $Y$ and $Z$, where $\Nat_*(\FK_Y^*,\FK_Z^*)$ denotes the set of graded natural transformations from the functor $\FK_Y^*$ to the functor $\FK_Z^*$.
The target category of $\FK^*$ is the category $\Mod(\NT_*)^{\mathbb Z/2}$ of graded modules over $\NT_*$, that is, $\Z/2$-graded additive functors $\NT_*\to \Ab^{\mathbb Z/2}$. Hence $\FK^*(A)$ consists of the groups $\FK_Y^*(A)$ together with the natural transformations $\FK_Y^*(A)\to\FK_Z^*(A)$.

For reasons of notation we will often find it convenient to consider instead the pre-additive category $\NT$ with objects $\LC(X)\times\{0,1\}$ and morphisms between $(Y,j)$ and $(Z,k)$ given by natural transformations
\[
\Nat(\FK_Y^j\FK_Z^k)\cong\KK_0(X;\Sus^kR_Z,\Sus^jR_Y),
\]
 where $\FK_Y^j(A)$ denotes $\K_j\bigl(A(Y)\bigr)$ for $j=0,1$ and $\Sus$ denotes suspension (with $\Sus^0A=A$). Let $\Mod(\NT)$ denote the category of modules over $\NT$, that is, additive functors $\NT\to\Ab$.

Given a graded $\NT_*$-module $M$, we define an $\NT$-module $D(M)$ as follows: we set $D(M)(Y,i) = M(Y)_i$ for $(Y,i)\in\LC(X)\times\{0,1\}$; for a morphism $f\colon (Y,i)\to (Z,j)$ in $\NT$, we define $D(M)(f)\colon D(M)(Y,i)\to D(M)(Z,j)$ as the composite
\[
M(Y)_i\hookrightarrow M(Y)_*\xrightarrow{M(f)} M(Z)_{*}\twoheadrightarrow M(Z)_j.
\]
It is straightforward to check that this yields a functor $D\colon\Mod(\NT_*)^{\mathbb Z/2}\to\Mod(\NT)$. In fact, $D$ is an equivalence of categories---an inverse can be defined by a direct sum construction. Consequently, we define the functor $\FK\colon\kk(X)\to\Mod(\NT)$ as the composite $\FK=D\circ\FK^*$.

\begin{definition}
  \label{def:generators}
Let $Y\in\LC(X)$, $U\subset Y$ be open in $Y$, and set $C=Y\setminus U$. 
A pair $(U,C)$ obtained in this way is called a \emph{boundary pair}.
The natural transformations occuring in the six-term exact sequence in $\K$\nb-the\-ory for the distinguished subquotient inclusion associated to $U\subset Y$ are denoted by $i_U^Y$, $r_Y^C$ and $\delta_C^Y$:
\[ \xymatrix{
\FK_U \ar[rr]^-{i_U^Y} && \FK_Y\ar[dl]^-{r_Y^C} \\
& \FK_C\ar[ul]|\circ^-{\delta_C^U} &
} \]
\end{definition}
These elements $i_U^Y$, $r_Y^C$ and $\delta_C^Y$ correspond to the $\KK(X)$-classes of the \Star{}ho\-mo\-mor\-phisms $R_Y\onto R_U$, $R_C\into R_Y$, and the extention $R_C\into R_Y\onto R_U$,  see~\cite{MN:Filtrated}.
These elements of $\NT_*$ satisfy the following relations. 
\begin{proposition}
  \label{pro:relations}
In the category $\NT_*$, the following relations hold. 
\begin{enumerate}[label=\textup{(\arabic*)}]
\item\label{it:identities} For every $Y\in\LC(X)$, \[i_Y^Y=r_Y^Y=\id_Y.\]

\item\label{it:biprodcut} 
If $Y,Z\in\LC(X)$ are topologically disjoint, then $Y\cup Z\in\LC(X)$ and
\[
 r_{Y\cup Z}^Y i_Y^{Y\cup Z} + r_{Y\cup Z}^Z i_Z^{Y\cup Z} =\id_{Y\cup Z}.
\]

\item\label{it:ii} For $Y\in\LC(X)$ and open subsets $U\subset V\subset Y$, 
\[ i_U^V i_V^Y  = i_U^Y.\]

\item\label{it:rr} For $Y\in\LC(X)$ and closed subsets $C\subset D\subset Y$, 
\[ r_Y^D r_D^C = r_Y^C.\]

\item\label{it:ir} For $Y\in\LC(X)$, an open subset $U\subset Y$ 
and a closed subset $C\subset Y$,
\[
 i_U^Y r_Y^C = r_U^{U\cap C} i_{U\cap C}^C.
\]

\item\label{it:id}
For a boundary pair $(U,C)$ in $X$ 
and an open subset $C'\subset C$, 
$(U,C')$ is a boundary pair and we have
\[
 i_{C'}^C \delta_C^U = \delta_{C'}^U.
\]

\item\label{it:dr}
For a boundary pair $(U,C)$ in $X$ 
and a closed subset $U'\subset U$, 
$(U',C)$ is a boundary pair and we have 
\[
\delta_C^U r_U^{U'} = \delta_C^{U'}.
\]

\item\label{it:rdi}
For $Y,Z,W \in \LC(X)$ such that 
$Y \cup W \in \LC(X)$ containing $Y,W$ as closed subsets, 
$Z \cup W \in \LC(X)$ containing $Z,W$ as open subsets, 
and $W \subset Y\cup Z$, 
we have 
\[
\delta_Y^{W\setminus Y} i_{W\setminus Y}^{Z} 
= r_Y^{W\setminus Z} \delta_{W\setminus Z}^{Z}.
\]
\end{enumerate}
\end{proposition}
\begin{proof}
We only prove \ref{it:rdi}, 
because the other relations can be proved 
similarly and more easily 
(their proofs can be found in \cite{bentmann}*{Section 3.2}).

Let us take $Y,Z,W \in \LC(X)$ as in \ref{it:rdi}. 
Let us also take a \csa\ $A$ over~$X$. 
Since both $Y$ and $W$ are closed 
subsets of $Y \cup W \in \LC(X)$, 
$Y \cap W$ is closed both in $Y$ and in $W$. 
Therefore we have a commutative diagram with exact rows
\[
\xymatrix{
0 \ar[r] & A(W \setminus Y)\ar@{=}[d]\ar[r] &  A(Y \cup W) \ar[r] \ar[d] 
& A(Y)\ar[r] \ar[d] & 0\phantom{.}\\
0 \ar[r] & A(W \setminus Y) \ar[r] &  A(W) \ar[r] 
& A(Y \cap W)\ar[r] & 0.
}
\]
Since both $Z$ and $W$ are open 
subsets of $Z \cup W \in \LC(X)$, 
$Z \cap W$ is open both in $Z$ and in $W$. 
Therefore we have a commutative diagram with exact rows
\[
\xymatrix{
0 \ar[r] & A(Z \cap W)\ar[d]\ar[r] &  A(W) \ar[r] \ar[d] 
& A(W \setminus Z)\ar[r] \ar@{=}[d] & 0\phantom{.}\\
0 \ar[r] & A(Z) \ar[r] &  A(Z \cup W) \ar[r] 
& A(W \setminus Z)\ar[r] & 0.
}
\]
From $W \subset Y\cup Z$, 
we get $W \setminus Y \subset Z \cap W$ 
and $W \setminus Z \subset Y \cap W$. 
Since $W \setminus Y$ is open in $W$, 
we see that $W \setminus Y$ is open in $Z \cap W$. 
Similarly, $W \setminus Z$ is closed in $Y \cap W$.
Hence we get a commutative diagram with exact rows
\[
\xymatrix{
0 \ar[r] & A(W \setminus Y) \ar[r] \ar[d] &  A(W)\ar@{=}[d] \ar[r] 
& A(Y \cap W)\ar[r] \ar[d] & 0\phantom{.}\\
0 \ar[r] & A(Z \cap W)\ar[r] &  A(W) \ar[r] 
& A(W \setminus Z)\ar[r] & 0.
}
\]
By combining these three diagrams, 
we obtain a commutative diagram with exact rows
\[
\xymatrix{
0 \ar[r] & A(W \setminus Y)\ar[d]\ar[r] &  A(Y \cup W) \ar[r] \ar[d] 
& A(Y)\ar[r] \ar[d] & 0\phantom{.}\\
0 \ar[r] & A(Z) \ar[r] &  A(Z \cup W) \ar[r] 
& A(W \setminus Z)\ar[r] & 0.
}
\]
From this digram, we get a commutative diagram 
\[
\xymatrix@C-1pt{
\K_*\bigl(A(Y \cup W)\bigr) \ar[d]^{ri} \ar[r]^r & \K_*\bigl(A(Y)\bigr) \ar[r]|\circ^\delta \ar[d]^r 
& \K_*\bigl(A(W \setminus Y)\bigr) \ar[r]^i \ar[d]^i 
& \K_*\bigl(A(Y \cup W)\bigr)\phantom{.} \ar[d]^{ri}\\
\K_*\bigl(A(Z \cup W)\bigr) \ar[r]^r & \K_*\bigl(A(W \setminus Z)\bigr) \ar[r]|\circ^\delta 
& \K_*\bigl(A(Z)\bigr) \ar[r]^i & \K_*\bigl(A(Z \cup W)\bigr). 
}
\]
Now \ref{it:rdi} follows from the commutativity of 
the middle square of this natural diagram. 
\end{proof}

\begin{remark}
From Proposition \ref{pro:relations}\ref{it:biprodcut}, 
we see that the empty set $\emptyset$ is a zero object in $\NT_*$
(because initial objects in pre-additive categories are also terminal).
From this and other relations in Proposition \ref{pro:relations}, 
we can conclude that compositions of consecutive maps in six-term sequences 
associated to relatively open subset inclusions vanish.
\end{remark}

\begin{remark}
We usually denote the even and the odd component of the element~$i_U^Y$ 
in $\NT_*$ defined in Definition~\ref{def:generators} simply by~$i_U^Y$.
Often, sub- and superscripts are suppressed when clear from context.
Similar comments apply to $r$ and $\delta$. 
\end{remark}

\begin{definition}
Let $\Catsixterm_*$ be the universal 
$\Z/2$-graded pre-additive category 
whose set of objects is $\LC(X)$ 
and whose set of morphisms are generated 
by elements as in Definition \ref{def:generators} 
with the relations as in Proposition \ref{pro:relations}.
Let $\Catsixterm$ be the corresponding pre-additive category 
with object set $\LC(X)\times\{0,1\}$. 
\end{definition}

By Proposition \ref{pro:relations}, 
we have a canonical additive functor $\Catsixterm\to\NT$. 
This functor has been shown to be an isomorphism 
in all examples which have been investigated---including accordion spaces
and all four-point spaces (see \cites{bentmann,MN:Filtrated}). 
However there is an example $Q$ of a finite $T_0$-space 
for which the functor $\Catsixterm\to\NT$ seems to be non-faithful 
(see Remark \ref{rem:Q}). 
For such spaces one would need to modify the definition of the category $\Catsixterm$, 
but we do not pursue this problem in this paper.

Let $\Fsixterm\colon\Mod(\NT)\to\Mod(\Catsixterm)$ be the functor 
induced by the canonical functor $\Catsixterm\to\NT$. 

\begin{definition} \label{def:FKsixterm}
We define \emph{concrete filtered $\K$\nb-the\-ory} $\FKsixterm\colon\kk(X)\to\Mod(\Catsixterm)$ as the composition $\Fsixterm\circ\FK$.
\end{definition}

\begin{remark}
As noted above, filtered $\K$\nb-the\-ory $\FK$ and concrete filtered $\K$\nb-the\-ory $\FKsixterm$ coincide for accordion spaces and all four-point spaces.
\end{remark}

\begin{definition}
An $\NT$-module $M$ is called \emph{exact} if, for all $Y\in\LC(X)$ and $U\in\Op(Y)$, the sequence
\[
\xymatrix{
M(U,0) \ar[r]^-{i} & M(Y,0) \ar[r]^-{r} & M(Y\setminus U, 0) \ar[d]^-{\delta} \\
M(Y\setminus U, 1) \ar[u]^-{\delta} & M(Y,1) \ar[l]^-{r} & M(U,1) \ar[l]^-{i}
} \]
is exact.
An $\NT$-module $M$ is called \emph{\rrzero{}} if, for all $Y\in\LC(X)$ and $U\in\Op(Y)$, the map $\delta\colon M(Y\setminus U,0)\to M(U,1)$ vanishes.

In the same way, we define exact $\Catsixterm$-modules and \rrzero{} $\Catsixterm$-mod\-ules.
\end{definition}

\begin{remark} \label{rem:K0lift}
For a \csa{} $A$ over~$X$, the module $\FK(A)$ is exact. It follows from \cite{Bentmann:Intermediate_cancellation}*{Lemma~3.4} that, if~$A$ is tight over~$X$, then $\FK(A)$ is \rrzero{} if and only if the underlying \csa{} of~$A$ is \emph{$\K_0$-liftable} in the sense of Pasnicu--R\o rdam~\cite{pasnicurordam}. By \cite{linrordam}*{Proposition~4}, all real-rank-zero \csa s are $\K_0$-liftable. By Theorem~4.2 and Example~4.8 of~\cite{pasnicurordam}, a tight, purely infinite \csa{} $A$ over~$X$ has real rank zero if and only if $\FK(A)$ is \rrzero{}. Analogous remarks apply with $\FKsixterm(A)$ in place of $\FK(A)$.
\end{remark}

The following theorem is the basis for the corollaries obtained in Section~\ref{sec:accordion}.

\begin{theorem}[{\cites{bentmann_koehler,MN:Filtrated,kirchberg}}] \label{bentmann}
Let $X$ be an accordion space. The canonical functor $\Catsixterm\to\NT$ is an isomorphism. Moreover, if $A$ and $B$ are stable Kirchberg $X$\nb-al\-ge\-bras with all simple subquotients in the bootstrap class, then any isomorphism $\FK(A)\to\FK(B)$ lifts to an $X$\nb-equi\-vari\-ant \Star{}isomorphism $A\to B$.
\end{theorem}

\section{Sheaves}
  \label{sec:sheaves}

In this section we introduce sheaves and cosheaves and recall that it suffices to specify them on a basis for the topology.

Let $X$ be an arbitrary topological space. 
Let $\Basis$ be a basis for the topology on~$X$. 
We note that the set $\Op$ of all open subsets is 
the largest basis for the topology on $X$. 
We also note that for a finite space $X$, 
the collection $\bigl\{\widetilde{\{x\}}\mid x\in X\bigr\}$
is an example of a basis. 
The set $\Basis$ is a category whose morphisms 
are inclusions. 

\begin{definition}
A \emph{covering} of a set $U \in \Basis$ is a collection $\{U_j\}_{j\in J} \subset \Basis$ such that $U_j\subseteq U$ for all $j\in J$ and $\bigcup_{j\in J} U_j=U$.
A \emph{presheaf on $\Basis$} is a contravariant functor $M\colon\Basis\to\Ab$. It is a \emph{sheaf on $\Basis$} if, for every 
$U \in \Basis$, every covering $\{U_j\}_{j\in J} \subset \Basis$ of~$U$, 
and all coverings $\{U_{jkl}\}_{l\in L_{jk}} \subset \Basis$ 
of $U_j\cap U_k$, the sequence
\begin{multline}
  \label{eq:sheaf_condition_on_basis}
0\longrightarrow M(U)
\xrightarrow{\left(M(i_{U}^{U_j})\right)}
\prod_{j\in J} M(U_j)
\xrightarrow{\left(M(i^{U_{jkl}}_{U_j})-M(i^{U_{jkl}}_{U_k})\right)}
\prod_{j,k\in J}\prod_{l\in L_{jk}} M(U_{jkl})
\end{multline}
is exact. A morphism for sheaves is a natural transformation of functors.
We denote by $\Sheaf(\Basis)$ the category of sheaves on $\Basis$.
\end{definition}

If $\Basis$ is closed under intersection 
(for example if $\Basis = \Op$), 
then the definition of sheaf can be replaced with 
the exactness of the sequence
\begin{multline*}
0\longrightarrow M(U)
\xrightarrow{\left(M(i_{U}^{U_j})\right)}
\prod_{j\in J} M(U_j)
\xrightarrow{\left(M(i^{U_j\cap U_k}_{U_j})-M(i^{U_j\cap U_k}_{U_k})\right)}
\prod_{j,k\in J} M(U_j\cap U_k)
\end{multline*}
for all $U \in \Basis$ and every covering $\{U_j\}_{j\in J} \subset \Basis$ of $U$.

\begin{lemma}
  \label{lem:sheaf_equivalence}
For a basis $\Basis$ for the topology on $X$, 
the restriction functor $\Sheaf(\Op)\to\Sheaf(\Basis)$ 
is an equivalence of categories.
\end{lemma}

\begin{proof}
This is a well-known fact in algebraic geometry (see, for instance the encyclopedic treatment in \cite{stacks-project}*{Lemma \href{http://math.columbia.edu/algebraic_geometry/stacks-git/locate.php?tag=009O}{009O}}). We confine ourselves on mentioning that \eqref{eq:sheaf_condition_on_basis} provides a formula for computing $M(U)$ for an arbitrary open subset~$U$.
\end{proof}

\begin{definition}
A \emph{precosheaf on $\Basis$} is a covariant functor $M\colon\Basis\to\Ab$. 
It is a \emph{cosheaf on $\Basis$} if, for every 
$U \in \Basis$, every covering $\{U_j\}_{j\in J} \subset \Basis$ of $U$, 
and all coverings $\{U_{jkl}\}_{l\in L_{jk}} \subset \Basis$ 
of $U_j\cap U_k$, the sequence
\begin{multline}
  \label{eq:cosheaf_condition_on_basis}
\bigoplus_{j,k\in J}\bigoplus_{l\in L_{jk}} M(U_{jkl})
\xrightarrow{\left(M(i_{U_{jkl}}^{U_j})-M(i_{U_{jkl}}^{U_k})\right)}
\bigoplus_{j\in J} M(U_j)
\xrightarrow{\left(M(i_{U_j}^U)\right)}
M(U) \longrightarrow 0.
\end{multline}
is exact. A morphism for cosheaves is a natural transformation of functors.
We denote by $\CoSheaf(\Basis)$ the category of cosheaves on $\Basis$.
\end{definition}

Similarly to the case of sheaves, 
if $\Basis$ is closed under intersection, 
the definition of cosheaf can be replaced with 
the exactness of the sequence
\begin{multline}
  \label{eq:cosheaf_condition_on_topology}
\bigoplus_{j,k\in J} M(U_j\cap U_k)
\xrightarrow{\left(M(i_{U_j\cap U_k}^{U_j})-M(i_{U_j\cap U_k}^{U_k})\right)}
\bigoplus_{j\in J} M(U_j)
\xrightarrow{\left(M(i_{U_j}^U)\right)}
M(U) \longrightarrow 0.
\end{multline}
for $U \in \Basis$ and a covering $\{U_j\}_{j\in J} \subset \Basis$ of $U$.

\begin{lemma}
  \label{lem:cosheaf_equivalence}
The restriction functor $\CoSheaf(\Op)\to\CoSheaf(\Basis)$ is an equivalence of categories.
\end{lemma}

\begin{proof}
This statement is the dual of Lemma \ref{lem:sheaf_equivalence} and follows in an analogous way. Again, \eqref{eq:cosheaf_condition_on_basis} can be used to compute $M(U)$ for an arbitrary open subset~$U$.
\end{proof}

With regard to the next section we remark that every finite $T_0$-space (more generally every Alexandrov space) comes with canonical bases for the open subsets, namely $\bigl\{\widetilde{\{x\}}\mid x\in X\bigr\}$, and for the closed subsets: $\bigl\{\overline{\{x\}}\mid x\in X\bigr\}$.

\begin{lemma}
  \label{lem:presheaf_on_canonical_base_is_sheaf}
Let $X$ be a finite $T_0$-space and let $S$ be a pre(co)sheaf on the basis $\Basis=\bigl\{\widetilde{\{x\}}\mid x\in X\bigr\}$. Then $S$ is a \textup{(}co\textup{)}sheaf.
\end{lemma}

\begin{proof}
This follows from the observation that in the basis $\Basis$ there are no non-trivial coverings, that is, if $\mathcal U$ is a covering of $U$, then $U\in\mathcal U$.
\end{proof}

\section{Filtered \texorpdfstring{$\K$}{K}-theory restricted to the canonical base}
  \label{sec:takeshi_invariant}
  
In this section, the functor $\FKtakeshi$ and the notions of unique path spaces and EBP spaces are introduced. The following lemma is straightforward to verify.

\begin{lemma} \label{upp}
For a finite $T_0$-space $X$ the following conditions are equivalent.
\begin{itemize}
\item For all $x,y \in X$, there is at most one path from $y$ to $x$. 
\item There are no elements $a,b,c,d$ in $X$ with $a<b<d$, $a<c<d$ and neither $b\leq c$ nor $c\leq b$.
\item For all $x,y\in X$ with $x\to y$, 
we have $\osi x \cup \csi y \in \LC(X)$.
\item For every boundary pair $(U,C)$, 
the pair $(\widetilde U, \overline C)$ is a boundary pair. 
\item For all $x\in X$, $\ob \{x\} = \displaystyle\bigsqcup_{y\to x} \osi y$.
\item For all $x\in X$, $\overline\partial \{x\} = \displaystyle\bigsqcup_{y\ot x} \csi y$.
\end{itemize}
\end{lemma}

\begin{definition}
A finite $T_0$-space $X$ is called a \emph{unique path space} if it satisfies the equivalent conditions specified in Lemma~\ref{upp}.
\end{definition}

Let $X$ be a unique path space.

\begin{definition}
Let $\Cattakeshi$ denote the universal pre-additive category generated by objects $\tcsi x$, $\gosi x$ for all $x\in X$ and morphisms $r_{\tcsi x}^{\tcsi y}$, $\delta_{\tcsi y}^{\gosi x}$ and $i_{\gosi x}^{\gosi y}$ when $x\to y$, subject to the relations
\begin{equation}
\sum_{x\to y}  r_{\tcsi x}^{\tcsi y} \delta_{\tcsi y}^{\gosi x} = \sum_{z\to x} \delta_{\tcsi x}^{\gosi z} i_{\gosi z}^{\gosi x}
\end{equation}
for all $x\in X$.
\end{definition}

\begin{lemma}
In the category $\Catsixterm$, we have the relation
\begin{equation*}
\sum_{x\to y}  r_{\overline{\{x\}}}^{\overline{\{y\}}} \delta_{\overline{\{y\}}}^{\widetilde{\{x\}}} = \sum_{z\to x} \delta_{\overline{\{x\}}}^{\widetilde{\{z\}}} i_{\widetilde{\{z\}}}^{\widetilde{\{x\}}}
\end{equation*}
for all $x\in X$.
\end{lemma}

\begin{proof}
Since $X$ is a unique path space, the collections $\Big(\overline{\{y\}}\Big)_{x\to y}$ and $\Big(\widetilde{\{z\}}\Big)_{z\to x}$ are disjoint, respectively. Hence the desired relation simplifies to
\begin{equation*}
r_{\overline{\{x\}}}^{\overline{\partial}\{x\}} \delta_{\overline{\partial}\{x\}}^{\widetilde{\{x\}}} =
\delta_{\overline{\{x\}}}^{\widetilde{\partial}\{x\}} i_{\widetilde{\partial}\{x\}}^{\widetilde{\{x\}}},
\end{equation*}
which follows from Proposition \ref{pro:relations}\ref{it:rdi} 
by setting $Y = \overline{\{x\}}$, $Z=\widetilde{\{x\}}$ 
and $W = \overline{\{x\}} \cup \widetilde{\{x\}}$.
\end{proof}
\begin{definition}
The previous lemma allows us to define an additive functor $\Cattakeshi\to\Catsixterm$ by $\tcsi x\mapsto (\overline{\{x\}},1)$ and $\gosi x\mapsto (\widetilde{\{x\}},0)$, and in the obvious way on morphisms.
Let \[\Ftakeshi\colon\Mod(\Catsixterm)\to\Mod(\Cattakeshi)\] denote the induced functor.
Define \emph{filtered $\K$\nb-the\-ory restricted to the canonical base}, 
$\FKtakeshi\colon\kk(X)\to\Mod(\Cattakeshi)$, 
as the composition of $\FKsixterm$ with $\Ftakeshi$.
\end{definition}
\begin{remark}
The invariant $\FKtakeshi$ is only defined for unique path spaces because the boundary map $\delta_{\csi y}^{\osi x}$ only exists when $\csi y\cup\osi x$ belongs to $\LC(X)$. 
We also point out that the invariant $\FKtakeshi$ can only be expected to be very useful for spaces such that the relation \eqref{eq:cond_of_delta} holds for all boundary pairs $(U,C)$.
\end{remark}

\begin{definition}
A $\Cattakeshi$-module $M$ is called \emph{exact} if the sequence
\begin{equation}
M(\tcsi x) \xrightarrow{\begin{pmatrix} r_{\tcsi x}^{\tcsi y} & -\delta_{\tcsi x}^{\gosi z} \end{pmatrix}} \bigoplus_{x\to y} M(\tcsi y) \oplus \bigoplus_{z\to x} M(\gosi z) \xrightarrow{\begin{pmatrix} \delta_{\tcsi y}^{\gosi x} \\ i_{\gosi z}^{\gosi x} \end{pmatrix}} M(\gosi x)
\end{equation}
is exact for all $x\in X$.
\end{definition}

\begin{lemma}
  \label{lem:takeshi_exact}
If $M$ is an exact $\Catsixterm$-module, then $\Ftakeshi(M)$ is an exact $\Cattakeshi$-module.
In particular, if $A$ is a \csa\ over~$X$, then the $\Cattakeshi$-module $\FKtakeshi(A)$ is exact.
\end{lemma}

\begin{proof}
Using again that the collections $\left(\overline{\{y\}}\right)_{x\to y}$ and $\left(\widetilde{\{z\}}\right)_{z\to x}$ are respectively disjoint, it suffices to prove exactness of the sequence
\begin{equation*}
M(\overline{\{x\}},1) \xrightarrow{\begin{pmatrix} r_{\overline{\{x\}}}^{\overline{\partial}\{x\}} & -\delta_{\overline{\{x\}}}^{\widetilde{\partial}\{x\}} \end{pmatrix}}
M(\overline{\partial}\{x\},1) \oplus M(\widetilde{\partial}\{x\},0) \xrightarrow{\begin{pmatrix} \delta_{\overline{\partial}\{x\}}^{\widetilde{\{x\}}} \\ i_{\widetilde{\partial}\{x\}}^{\widetilde{\{x\}}} \end{pmatrix}} M(\widetilde{\{x\}},0),
\end{equation*}
which follows from a diagram chase through the commutative diagram
\[
\xymatrix{
M(\{x\},1)
\ar[r]\ar@{=}[d]
&M(\overline{\{x\}},1)
\ar[r]
\ar[d]|\circ
& M(\overline{\partial}\{x\},1)
\ar[r]|\circ\ar[d]|\circ
&M(\{x\},0)\ar@{=}[d]
\\
M(\{x\},1)\ar[r]|\circ
&M(\widetilde{\partial}\{x\},0)
\ar[r]
& M(\widetilde{\{x\}},0)\ar[r]
&M(\{x\},0)
}
\]
whose rows are exact.
\end{proof}

\begin{definition}
Let $X$ be a finite $T_0$-space. A boundary pair $(U,C)$ in~$X$ is called
\emph{elementary} if $U$ and $C$ are connected, 
$U$ is open, $C$ is closed
and if, moreover, $U\subset\widetilde C$ and $C\subset\overline U$.
\end{definition}

\begin{definition}
  \label{def:EBP}
A unique path space $X$ is called an \emph{EBP space} if every elementary boundary pair $(U,C)$ in $X$ is of the form $(\widetilde{\{x\}},\overline{\{y\}})$ for two points~$x$ and~$y$ in~$X$ with $x\to y$.
\end{definition}

\begin{lemma} \label{toplemma}
Let $X$ be an EBP space, and let $(U,C)$ be a boundary pair in $X$.  Then the following relation holds in the category $\Catsixterm_*$:
\begin{equation}
  \label{eq:cond_of_delta}
  \delta_C^U=
  \sum_{x\to y, x\in U, y\in C} r_C^{\overline{\{y\}}\cap C}\ i_{\overline{\{y\}}\cap C}^{\overline{\{y\}}}\ \delta_{\overline{\{y\}}}^{\widetilde{\{x\}}}\ r_{\widetilde{\{x\}}}^{\widetilde{\{x\}}\cap U}\ i_{\widetilde{\{x\}}\cap U}^U .
\end{equation}
\end{lemma}
\begin{proof}
We would like to show the relation \eqref{eq:cond_of_delta}
for a boundary pair $(U,C)$ in $X$.
The proof goes by the induction on the number 
$|U \cup C|$ 
of elements of $U \cup C$. 
If either $U$ or $C$ is empty, then both sides of 
\eqref{eq:cond_of_delta} are $0$. 
This takes care of the case $|U \cup C| = 0$. 
Suppose for a natural number $n$, 
we have shown \eqref{eq:cond_of_delta} 
for all boundary pairs $(U,C)$ with $|U \cup C| \leq n$, 
and take a boundary pair $(U,C)$ with $|U \cup C| = n+1$, arbitrarily. 
We are going to show \eqref{eq:cond_of_delta} for this pair. 
If either $U$ or $C$ is empty, again both sides of \eqref{eq:cond_of_delta} are zero.
So we may assume that both $U$ and $C$ are non-empty. 
Suppose $U$ is not connected, 
and choose two non-empty open and closed subsets $U_1$ and $U_2$ 
of $U$ such that $U = U_1 \sqcup U_2$. 
Then for $i=1,2$, $(U_i,C)$ is a boundary pair 
with $|U_i \cup C| \leq n$. 
Thus by the assumption of the induction, 
both $(U_1,C)$ and $(U_2,C)$ satisfy \eqref{eq:cond_of_delta}. 
Hence by \ref{it:biprodcut}, \ref{it:dr} and \ref{it:ii} of Proposition~\ref{pro:relations} 
we have 
\begin{align*}
\delta_C^U
& = \delta_C^U (r_{U}^{U_1}i_{U_1}^{U} + r_{U}^{U_2}i_{U_2}^{U}) \\
& = \delta_C^{U_1}i_{U_1}^{U} + \delta_C^{U_2}i_{U_2}^{U} \\
& = \Big(\sum_{x\to y, x\in U_1, y\in C} r_C^{\overline{\{y\}}\cap C}\
i_{\overline{\{y\}}\cap C}^{\overline{\{y\}}}\
\delta_{\overline{\{y\}}}^{\widetilde{\{x\}}}\
r_{\widetilde{\{x\}}}^{\widetilde{\{x\}}\cap U_1}\ i_{\widetilde{\{x\}}\cap
U_1}^{U_1}\Big)i_{U_1}^{U} \\
&\phantom{==} 
+ \Big(\sum_{x\to y, x\in U_2, y\in C} r_C^{\overline{\{y\}}\cap C}\
i_{\overline{\{y\}}\cap C}^{\overline{\{y\}}}\
\delta_{\overline{\{y\}}}^{\widetilde{\{x\}}}\
r_{\widetilde{\{x\}}}^{\widetilde{\{x\}}\cap U_2}\ i_{\widetilde{\{x\}}\cap
U_2}^{U_2}\Big)i_{U_2}^{U} \\
& = \sum_{x\to y, x\in U, y\in C} r_C^{\overline{\{y\}}\cap C}\
i_{\overline{\{y\}}\cap C}^{\overline{\{y\}}}\
\delta_{\overline{\{y\}}}^{\widetilde{\{x\}}}\
r_{\widetilde{\{x\}}}^{\widetilde{\{x\}}\cap U}\ i_{\widetilde{\{x\}}\cap
U}^U 
\end{align*}
since we for $x \in U_i$ 
 have $\widetilde{\{x\}}\cap U_1 = \widetilde{\{x\}}\cap U$ 
because $U_i \subset U$ is open. 
This shows \eqref{eq:cond_of_delta} for $(U,C)$. 
Thus we may now assume $U$ is connected. 
In a very similar way, 
we get \eqref{eq:cond_of_delta} 
using the assumption of the induction if $C$ is not connected. 
Thus we may assume $C$ is connected. 
Next suppose we have $U\not\subset\widetilde C$. 
Set $U' = U \cap \widetilde C$ which is a proper open subset of $U$. 
The pair $(U',C)$ is a boundary pair 
because $U' \cup C = (U \cup C) \cap \widetilde C \in \LC(X)$. 
We have 
$\delta_C^U = \delta_{C}^{U'} i_{U'}^{U}$
by applying \ref{it:rdi} of Proposition~\ref{pro:relations} 
for $Y=C$, $Z=U$ and $W= U' \cup C$. 
Since $|U' \cup C| \leq n$, we get by the assumption of the induction  that
\begin{align*}
\delta_C^U
&= \delta_{C}^{U'} i_{U'}^{U} \\
&= \Big(\sum_{x\to y, x\in U', y\in C} r_C^{\overline{\{y\}}\cap C}\
i_{\overline{\{y\}}\cap C}^{\overline{\{y\}}}\
\delta_{\overline{\{y\}}}^{\widetilde{\{x\}}}\
r_{\widetilde{\{x\}}}^{\widetilde{\{x\}}\cap U'}\ i_{\widetilde{\{x\}}\cap
U'}^{U'}\Big) i_{U'}^{U} \\
&= \sum_{x\to y, x\in U, y\in C} r_C^{\overline{\{y\}}\cap C}\
i_{\overline{\{y\}}\cap C}^{\overline{\{y\}}}\
\delta_{\overline{\{y\}}}^{\widetilde{\{x\}}}\
r_{\widetilde{\{x\}}}^{\widetilde{\{x\}}\cap U}\ i_{\widetilde{\{x\}}\cap
U}^{U}
\end{align*}
since $x\to y$, $x\in U$ and $y\in C$ 
imply $x \in U'$, 
and we have $\widetilde{\{x\}}\cap U = \widetilde{\{x\}}\cap U'$. 
This shows \eqref{eq:cond_of_delta} for $(U,C)$. 
Thus we may now assume $U\subset\widetilde C$. 
In a very similar way, 
we get \eqref{eq:cond_of_delta} 
using the assumption of the induction if $C\not\subset\overline U$. 
Thus we may assume $C\subset\overline U$.

It remains to show \eqref{eq:cond_of_delta} 
for a boundary pair $(U,C)$ such that $U$ and $C$ are connected, 
$U\subset\widetilde C$ and $C\subset\overline U$.
To this end, we use the assumption of the lemma. 
Take such a pair $(U,C)$. 
Since $X$ is a unique path space, 
the pair $(\widetilde U, \overline C)$ is a boundary pair 
by Lemma~\ref{upp}. 
It is not difficult to see 
that the pair $(\widetilde U, \overline C)$ is elementary. 
Hence by the assumption of the lemma, 
there exist $x \in \widetilde U$ and $y \in \overline C$ 
such that $\widetilde U = \widetilde{\{x\}}$, 
$\overline C = \overline{\{y\}}$ and $x\to y$.
By \ref{it:id} and \ref{it:dr} of Proposition~\ref{pro:relations}, 
we get 
\begin{align*}
\delta_C^U
&= i_{C}^{\overline C} \delta_{\overline C}^{\widetilde U} r_{\widetilde U}^{U}
= i_{C}^{\overline{\{y\}}} \delta_{\overline{\{y\}}}^{\widetilde{\{x\}}} 
r_{\widetilde{\{x\}}}^{U}. 
\end{align*}
It remains to prove that $(x,y)$ 
is the only pair satisfying $x\to y$, $x\in U$ and $y\in C$. 
First note that $\widetilde U = \widetilde{\{x\}}$ 
implies $x\in U$, and also that 
$\overline C = \overline{\{y\}}$ implies $y\in C$. 
Now take $u \in U$ and $c \in  C$ with $u \to c$. 
Since $U \subset \widetilde{\{x\}}$ and $C \subset \overline{\{y\}}$, 
there exist a path from $u$ to $x$, 
and a path from $y$ to $c$. 
These two paths together with the arrow $x \to y$ 
give us a path from $u$ to $c$. 
Since $X$ is a unique path space, this path should coincide 
with the arrow $u \to c$. 
Hence we get $u=x$ and $c=y$. 
This finishes the proof. 
\end{proof}

\begin{lemma}
Let $X$ be a finite $T_0$-space. 
Assume that the directed graph associated to $X$ is a \emph{forrest}, 
that is, it contains no undirected cycles. Then $X$ is an EBP space.
\end{lemma}

\begin{proof}
It is clear that, if the directed graph associated to $X$ is a forrest, then $X$ is a unique path space. 
Let us take an elementary boundary pair $(U,C)$. 
Choose a minimal element $x\in U$. 
Since $U\subset \widetilde C$, there is $y\in C$ with $x > y$. 
We can, moreover, assume that $x\to y$ 
because $U\cup C$ is locally closed and $x$ is minimal in $U$. 
Since $U$ is open and $C$ is closed, 
we have $\widetilde{\{x\}}\subset U$ and $\overline{\{y\}}\subset C$. 
We will show that these inclusions are equalities 
using the fact that $X$ is a forrest. 
Take $u\in U$ arbitrarily. 
Since $U\subset \widetilde C$, 
there exists an element $c \in C$ such that $u > c$. 
Thus we have a path from $u$ to $c$. 
Since both $U$ and $C$ are connected, 
there exist undirected paths from $u$ to $x$ 
and from $y$ to $c$. 
These two paths give us an undirected path 
from $u$ to $c$ through the arrow $x \to y$. 
This path should coincide with 
the directed path from $u$ to $c$ 
because $X$ contains no undirected cycles. 
Hence we get a path from $u$ to $x$. 
This shows $u \in \widetilde{\{x\}}$, 
and therefore we get $U = \widetilde{\{x\}}$. 
In a similar manner, we get $C = \overline{\{y\}}$. 
\end{proof}

\begin{remark} \label{rem:Q}
The above lemma applies, in particular, to accordion spaces. The conclusion of Lemma~\ref{toplemma} can also be verified for various unique path spaces which are not forrests---the smallest example being the so-called pseudocircle with four points. Consider, however, the sixteen-point space $Q$ defined by the directed graph
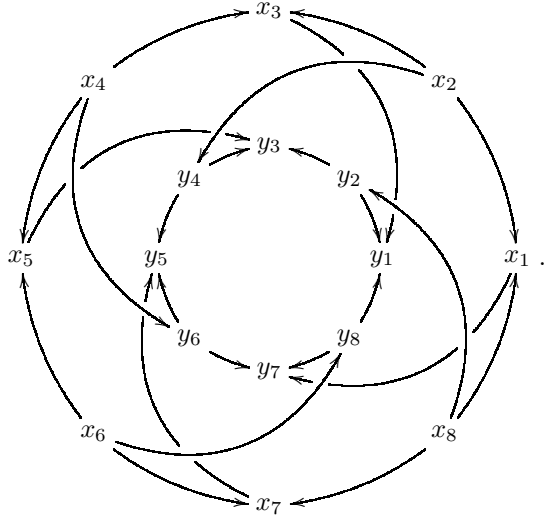
\begin{equation*}
\begin{split}
\xy
(15,0)*+{y_1}="E1"; (0,15)*+{y_3}="E2";
(-15,0)*+{y_5}="E3"; (0,-15)*+{y_7}="E4";
(10.6066,10.6066)*+{y_2}="F1"; (-10.6066,10.6066)*+{y_4}="F2";
(-10.6066,-10.6066)*+{y_6}="F3"; (10.6066,-10.6066)*+{y_8}="F4";
(33,0)*+{x_1}="V1"; (0,33)*+{x_3}="V2";
(-33,0)*+{x_5}="V3"; (0,-33)*+{x_7}="V4";
(23.33452,23.33452)*+{x_2}="E5"; (-23.33452,23.33452)*+{x_4}="E6";
(-23.33452,-23.33452)*+{x_6}="E7"; (23.33452,-23.33452)*+{x_8}="E8";
{\ar@/^0.2pc/ "F1"; "E1"}; {\ar@/_0.2pc/ "F1"; "E2"};
{\ar@/^0.2pc/ "F2"; "E2"}; {\ar@/_0.2pc/ "F2"; "E3"};
{\ar@/^0.2pc/ "F3"; "E3"}; {\ar@/_0.2pc/ "F3"; "E4"};
{\ar@/^0.2pc/ "F4"; "E4"}; {\ar@/_0.2pc/ "F4"; "E1"};
{\ar@/^/ "E5"; "V1"}; 
{\ar@/^/ "E6"; "V2"}; 
{\ar@/^/ "E7"; "V3"};
{\ar@/^/ "E8"; "V4"}; 
{\ar@/_/ "E5"; "V2"};
{\ar@/_/ "E6"; "V3"};
{\ar@/_/ "E7"; "V4"};
{\ar@/_/ "E8"; "V1"};
{\ar@/^2pc/ |(.3)\hole |(.85)\hole "V3"; "E2"};
{\ar@/^2pc/ |(.3)\hole |(.85)\hole "V2"; "E1"};
{\ar@/^2pc/ |(.3)\hole |(.85)\hole "V1"; "E4"};
{\ar@/^2pc/ |(.3)\hole |(.85)\hole "V4"; "E3"};
{\ar@/_2pc/ "E5"; "F2"};
{\ar@/_2pc/ "E6"; "F3"};
{\ar@/_2pc/ "E7"; "F4"};
{\ar@/_2pc/ "E8"; "F1"};
\endxy
\end{split} .
\end{equation*}
Then $Q$ is a unique path space that is not an EBP space because the subsets $U=\{x_1,x_2,\ldots,x_8\}$ and $C=\{y_1,y_2,\ldots,y_8\}$ give an elementary boundary pair $(U,C)$ that does not satisfy $U=\osi x$ nor $C=\csi y$ for any $x,y\in X$.
A simple computation shows that the boundary decomposition \eqref{eq:cond_of_delta} of $\delta_C^U$ indeed holds in the category $\NT_*$. However, we believe that it does not hold in the category $\Catsixterm_*$.
\end{remark}

The following theorem has two important consequences.
Firstly, as stated in Corollary~\ref{cor:sixterm_to_takeshi}, it implies that for real-rank-zero \csa s, isomorphisms on~$\FKtakeshi$ lift to isomorphisms on~$\FKsixterm$.
By Theorem~\ref{bentmann}, $\FKsixterm$ is strongly complete for stable Kirchberg $X$\nb-al\-ge\-bras when $X$ is an accordion space.
Secondly, by Lemma~5.6 of~\cite{bentmann_koehler}, if $X$ is an accordion space, any exact $\NT$-module is of the form $\FK(A)$ for some Kirchberg $X$\nb-al\-ge\-bra~$A$, so any exact $\Cattakeshi$-module is of the form $\FKtakeshi(A)$ for some Kirchberg $X$\nb-al\-ge\-bra~$A$ of real rank zero.
This second consequence is useful for constructing examples of Kirchberg $X$\nb-al\-ge\-bras.

\begin{theorem}
  \label{thm:sixterm_to_takeshi}
  Let $X$ be an EBP space.
The functor
\[
\Ftakeshi\colon\Mod(\Catsixterm)\to\Mod(\Cattakeshi)
\]
restricts to an equivalence between the category of exact \rrzero{} $\Catsixterm$-modules and the category of exact $\Cattakeshi$-modules.
\end{theorem}

A proof of this theorem is given after the following remark and corollary.

\begin{remark}
The proof of Theorem \ref{thm:sixterm_to_takeshi} given below works in fact not only for EBP spaces but more generally for unique path spaces for which the relation \eqref{eq:cond_of_delta} holds in the category $\Catsixterm$ for all boundary pairs $(U,C)$, see~Lemma~\ref{toplemma} and Remark~\ref{rem:Q}.
\end{remark}

\begin{corollary} \label{cor:sixterm_to_takeshi}
Let $A$ and $B$ be \csa s of real rank zero over an EBP space~$X$.  Then for any homomorphism $\phi\colon\FKtakeshi(A)\to\FKtakeshi(B)$, there exists a unique homomorphism $\Phi\colon\FKsixterm(A)\to\FKsixterm(B)$ such that $\Ftakeshi(\Phi)=\phi$.  If $\phi$ is an isomorphism, then so is $\Phi$.
\end{corollary}

\begin{proof}[Proof of Theorem \textup{\ref{thm:sixterm_to_takeshi}}]

We will explicitly define a functor from the category of exact $\Cattakeshi$-modules to the category of exact \rrzero{} $\Catsixterm$-modules.

Let an exact $\Cattakeshi$-module $N$ be given. We will define an $\Catsixterm$-module $M$. We begin in the obvious way: For $x\in X$, let $M(\overline{\{x\}},1)=N(\tcsi x)$ and $M(\widetilde{\{x\}},0)=N(\gosi x)$. Similarly, for $x\to y$, we define the even component of $i_{\widetilde{\{x\}}}^{\widetilde{\{y\}}}$ to be $i_{\gosi x}^{\gosi y}$, the odd component of $r_{\overline{\{x\}}}^{\overline{\{y\}}}$ to be  $r_{\tcsi x}^{\tcsi y}$, and the odd-to-even component of $\delta_{\overline{\{y\}}}^{\widetilde{\{x\}}}$ to be $\delta_{\tcsi y}^{\gosi x}$. This makes sure that, finally, we will have $\Ftakeshi(M)=N$. Also, we of course define $\delta_C^U\colon M(C,0)\to M(U,1)$ to be zero for every boundary pair $(U,C)$ so that $M$ will be \rrzero{}.

For $x\geq y$, let $x\to x_1\to x_2\to\cdots\to x_n\to y$ be the unique path from $x$ to $y$. Define the even component of $i_{\widetilde{\{x\}}}^{\widetilde{\{y\}}}$ to be the composition $i_{\gosi x}^{\gosi {x_1}} i_{\gosi {x_1}}^{\gosi {x_2}}\cdots i_{\gosi {x_n}}^{\gosi y}$ and the odd component of $r_{\overline{\{x\}}}^{\overline{\{y\}}}$ as the composition $r_{\tcsi x}^{\tcsi{x_1}} r_{\tcsi{x_1}}^{\tcsi{x_2}}\cdots r_{\tcsi{x_n}}^{\tcsi{y}}$. In case of $x=y$, this specifies to $i_{\gosi x}^{\gosi {x}}=\id_{M(\widetilde{\{x\}},0)}$ and $r_{\tcsi x}^{\tcsi x}=\id_{M(\overline{\{x\}},1)}$. If we have $x\to y$, then these definitions coincide with the ones we gave before.

We observe that the groups $M(\widetilde{\{x\}},0)$ with the maps $i_{\widetilde{\{x\}}}^{\widetilde{\{y\}}}$ define a precosheaf on $\Basis=\left\{\widetilde{\{x\}}\mid x\in X\right\}$. By Lemma \ref{lem:presheaf_on_canonical_base_is_sheaf} it is in fact a cosheaf. We can therefore apply Lemma \ref{lem:cosheaf_equivalence} and obtain groups $M(U,0)$ for all sets $U$ and maps $i_U^V\colon M(U,0)\to M(V,0)$ for open sets $U\subset V$ which fulfill the relations \ref{it:identities} and \ref{it:ii} in Proposition \ref{pro:relations}.

For an arbitrary locally closed subset $Y\in\LC(X)$ we write $Y=V\setminus U$ with open sets $U\subset V$ and define $M(Y,0)$ as the cokernel of the map
$
i_{U}^{V}\colon M(U,0)\to M(V,0)
$.
That this definition does not depend on the choice of $U$ and $V$ can be seen in a way similar to the proof of \cite{MN:Bootstrap}*{Lemma 2.15} using that pushouts of abelian groups preserve cokernels. We obtain maps $r_V^Y\colon M(V,0)\to M(Y,0)$ for every open set $V$ with relatively closed subset $Y\subseteq V$ such that the following holds:
If $Y\in\LC(X)$ can be written as differences $V_i\setminus U_i$ of open sets for $i\in\{1,2\}$ such that $U_1\subseteq U_2$ and $V_1\subseteq V_2$, then the diagram
\begin{equation}
  \label{eq:r_def}
\begin{split}
 \xymatrix{
M(U_1,0)\ar[r]^i\ar[d]^i & M(V_1,0)\ar@{->>}[r]^r\ar[d]^i & M(Y,0)\ar@{=}[d] \\
 M(U_2,0)\ar[r]^i & M(V_2,0)\ar@{->>}[r]^r & M(Y,0)
}
\end{split}
\end{equation}
commutes.

For a relatively open subset $U\subset Y\in\LC(X)$ we obtain a unique map $i_U^Y\colon M(U,0)\to M(Y,0)$ using the diagram
\begin{equation}
  \label{eq:r_relation}
\begin{split}
\xymatrix{
M(\widetilde\partial U,0)\ar[r]^-i\ar[d]^-i & M(\widetilde U,0)\ar@{->>}[r]^-r\ar[d]^-i & M(U,0)\ar@{..>}[d]^-i \\
M(\widetilde\partial Y,0)\ar[r]^-i & M(\widetilde Y,0)\ar@{->>}[r]^-r & M(Y,0).
}
\end{split}
\end{equation}
It is easy to check that this map coincides with the previously defined one in case $Y$ is open.

We find that, for $Y_i\in\LC(X)$ with $Y_1\subset Y_2$ open, and $Y_i=V_i\setminus U_i$ for $i\in\{1,2\}$ and open sets $U_i$, $V_i$ such that $U_1\subseteq U_2$ and $V_1\subseteq V_2$, the diagram
\begin{equation}
  \label{eq:ri_relations_for_opens}
\begin{split}
 \xymatrix{
M(U_1,0)\ar[r]^i\ar[d]^i & M(V_1,0)\ar@{->>}[r]^r\ar[d]^i & M(Y_1,0)\ar[d]^i \\
 M(U_2,0)\ar[r]^i & M(V_2,0)\ar@{->>}[r]^r & M(Y_2,0)
}
\end{split}
\end{equation}
commutes. We know this already for the left-hand square. For the right-hand square, it can be seen as follows: since $V_1$ is covered by $U_1$ and $\widetilde{Y_1}$, it suffices to check commutativity on the images $i_{U_1}^{V_1}\bigl(M(U_1,0)\bigr)$ and $i_{\widetilde{Y_1}}^{V_1}\bigl(M(\widetilde{Y_1},0)\bigr)$. On $i_{U_1}^{V_1}\bigl(M(U_1,0)\bigr)$ both compositions vanish. On the image of $M(\widetilde{Y_1},0)$, commutativity follows from \eqref{eq:r_def} and \eqref{eq:r_relation} considering the diagram
\[
\xymatrix{
M(\widetilde{Y_1},0)\ar[r]_i\ar[d]^i\ar@/^1.2pc/[rr]^r & M(V_1,0)\ar@{->>}[r]_r\ar[d]^i & M(Y_1,0)\ar[d]^i \\
 M(\widetilde{Y_2},0)\ar[r]^i\ar@/_1.2pc/[rr]_r & M(V_2,0)\ar@{->>}[r]^r & M(Y_2,0) .
}
\]

Now let $Y\in\LC(X)$, let $U$ be a relatively open subset of $Y$ and let $C=Y\setminus U$. Consider the diagram
\begin{equation}
  \label{eq:def_of_arbitrary_r}
\begin{split}
\xymatrix{
M(\widetilde\partial U,0)\ar[r]^-i\ar[d]^-i & M(\widetilde U,0)\ar@{->>}[r]^-r\ar[d]^-i & M(U,0)\ar[d]^-i \\
M(\widetilde\partial Y,0)\ar[r]^-i\ar@{->>}[d]^-r & M(\widetilde Y,0)\ar@{->>}[r]^-r\ar@{->>}[d]^-r & M(Y,0)\ar@{..>>}[d] \\
M(\widetilde\partial Y\setminus\widetilde\partial U,0)\ar[r]^-i & M(\widetilde Y\setminus\widetilde U,0)\ar@{->>}[r]^-r & M(C,0),
} 
\end{split}
\end{equation}
whose solid squares commute and whose rows and solid columns are exact. A diagram chase shows that there is a unique surjective map $r_Y^C\colon M(Y,0)\to M(C,0)$, as indicated by the dotted arrow, making the bottom-right square commute and making the right-hand column exact at $M(Y,0)$. Again, we can easily check that this map coincides with the previously defined one in case $Y$ is open.

We have now defined the even part of the module $M$ completely. It is straight-forward to check the relations \ref{it:ii} and \ref{it:rr} in Proposition \ref{pro:relations}. We will now prove that the relation \ref{it:ir} holds as well.

For this purpose, fix $Y\in\LC(X)$, let $U\subset Y$ be open and let $C\subset Y$ be closed. Consider the diagram
  \[
 \xymatrix{
 M(\widetilde U,0)\ar@{->>}[r]^-r\ar[d]^i & M(U,0)\ar[r]^-r\ar[d]^i  & M(U\cap C,0)\ar[d]^i \\
 M(\widetilde Y,0)\ar[r]^-r & M(Y,0)\ar[r]^-r  & M(C,0)
 } 
 \]
We would like to prove that the  right-hand square commutes. The  left-hand square commutes by definition of the map $i_U^Y$. Since $\widetilde U\cap C=U\cap C$, we can therefore assume without loss of generality that $U$ and $Y$ are open. Commutativity then follows from \eqref{eq:ri_relations_for_opens}.

Next, we will convince ourselves that the relation \ref{it:biprodcut} in Proposition \ref{pro:relations} holds on the even part of $M$. Let $W=Y\sqcup Z$ be a topologically disjoint union of subsets $Y,Z\in\LC(X)$. Fix $w\in M(W,0)$. Then $(w-w r_W^Z i_Z^W)r_W^Z=0$ as $i_Z^Wr_W^Z=\id_Z$. Hence there is $y\in M(Y,0)$ with $y i_Y^W=w-w r_W^Z i_Z^W$. Applying $r_W^Y$ shows $y=w r_W^Y$ as $i_Z^Wr_W^Y=0$. We get
\[
 w (r_W^Y i_Y^W + r_W^Z i_Z^W) = y i_Y^W + w r_W^Z i_Z^W = w.
\]
We have shown that $r_W^Y i_Y^W + r_W^Z i_Z^W=\id_W$ as desired.

We have defined all even groups for the desired module $M$ and the action of all transformations between them. We have checked all relations only involving transformations between even groups and verified exactness of $M(C,0)\to M(Y,0)\to M(U,0)$ for every boundary pair $Y=U\cup C$.

We intend to do the same for the odd part of the module $M$ in an analogous way. We start out with the given data consisting of the groups $M(\overline{\{x\}},1)$ and the maps $r_{\tcsi x}^{\tcsi y}$, $x\to y$, extend this to a sheaf on the basis $\bigl\{\overline{\{x\}}\mid x\in X\bigr\}$ of closed sets and apply Lemma \ref{lem:sheaf_equivalence}. Observing that every locally closed subset of $X$ can be written as a difference of two nested closed sets and using the functoriality of the kernel of a group homomorphism, we define groups $M(\overline{Y},1)$ for all $Y\in\LC(X)$ and actions for all transformations between these odd groups. Using arguments analogous to the ones above, we can verify the relations \ref{it:identities} to \ref{it:ir} in Proposition \ref{pro:relations} on the odd part of $M$. 

It remains to define the odd-to-even components of the boundary maps $\delta_C^U$ for all boundary pairs $(U,C)$, which has only been done in the special case $U=\widetilde{\{x\}}$, $C=\overline{\{y\}}$ with $x\to y$. Our general definition for $\delta_C^U\colon M(C,1)\to M(U,0)$ is
\begin{equation}
  \label{eq:def_of_delta}
  \delta_C^U=
  \sum_{x\to y, x\in U, y\in C} r_C^{\overline{\{y\}}\cap C}\ i_{\overline{\{y\}}\cap C}^{\overline{\{y\}}}\ \delta_{\overline{\{y\}}}^{\widetilde{\{x\}}}\ r_{\widetilde{\{x\}}}^{\widetilde{\{x\}}\cap U}\ i_{\widetilde{\{x\}}\cap U}^U.
\end{equation}
Our next aim is to verify the relations \ref{it:id}, \ref{it:dr} and \ref{it:rdi} in Proposition \ref{pro:relations}. 
We begin with relation \ref{it:id}. Let $(U,C)$ be a boundary pair and let $C'\subset C$ be relatively open. We have by the relations \ref{it:ii} and \ref{it:ir} that
\begin{align*}
 i_{C'}^C\ \delta_C^U &= i_{C'}^C \Big( \sum_{x\to y, x\in U, y\in C} r_C^{\overline{\{y\}}\cap C}\ i_{\overline{\{y\}}\cap C}^{\overline{\{y\}}}\ \delta_{\overline{\{y\}}}^{\widetilde{\{x\}}}\ r_{\widetilde{\{x\}}}^{\widetilde{\{x\}}\cap U}\ i_{\widetilde{\{x\}}\cap U}^U \Big) \\
&= \sum_{x\to y, x\in U, y\in C} r_{C'}^{\overline{\{y\}}\cap C'}\ i_{\overline{\{y\}}\cap C'}^{\overline{\{y\}}}\ \delta_{\overline{\{y\}}}^{\widetilde{\{x\}}}\ r_{\widetilde{\{x\}}}^{\widetilde{\{x\}}\cap U}\ i_{\widetilde{\{x\}}\cap U}^U.
\end{align*}
Since $C'$ is relatively open in $C$, $\overline{\{y\}}\cap C'$ is empty unless $y\in C'$. Therefore, the above sum equals
\[
 \delta_{C'}^U=\sum_{x\to y, x\in U, y\in C'} r_{C'}^{\overline{\{y\}}\cap C'}\ i_{\overline{\{y\}}\cap C'}^{\overline{\{y\}}}\ \delta_{\overline{\{y\}}}^{\widetilde{\{x\}}}\ r_{\widetilde{\{x\}}}^{\widetilde{\{x\}}\cap U}\ i_{\widetilde{\{x\}}\cap U}^U.
\]
This shows relation \ref{it:id}. The relation \ref{it:dr} follows similarly.

Next we will check relation \ref{it:rdi}. 
Let $Y,Z,W \in \LC(X)$ such that 
$Y \cup W \in \LC(X)$ containing $Y,W$ as closed subsets, 
$Z \cup W \in \LC(X)$ containing $Z,W$ as open subsets, 
and $W \subset Y\cup Z$. 
For each $x \in Z$ and $y \in Y$ with $x \to y$, 
we define $\gamma_{x,y}\colon M(Y,1) \to M(Z,0)$ by 
\[
\gamma_{x,y} = 
r_{Y}^{\overline{\{y\}}\cap Y}\ i_{\overline{\{y\}}\cap Y}^{\overline{\{y\}}}\ \delta_{\overline{\{y\}}}^{\widetilde{\{x\}}}\ r_{\widetilde{\{x\}}}^{\widetilde{\{x\}}\cap Z}\ i_{\widetilde{\{x\}}\cap Z}^Z.
\]
Since $W \setminus Y$ is an open subset of $Z$ 
(see the proof of Proposition~\ref{pro:relations}), 
we have $\widetilde{\{x\}}\cap (W \setminus Y) = \widetilde{\{x\}}\cap Z$ 
for each $x \in W \setminus Y$. 
We also have $y \in W$ if $y \in Y$ satisfies $x \to y$ 
for some $x \in W \setminus Y$ because $W \subset Y \cup W$ is closed. 
Therefore, by the relation \ref{it:rr} we get 
\begin{align*}
\delta_Y^{W\setminus Y} i_{W\setminus Y}^{Z} 
& = \Big(
\sum_{x\to y, x\in W \setminus Y, y\in Y} r_Y^{\overline{\{y\}}\cap Y}\ i_{\overline{\{y\}}\cap Y}^{\overline{\{y\}}}\ \delta_{\overline{\{y\}}}^{\widetilde{\{x\}}}\ r_{\widetilde{\{x\}}}^{\widetilde{\{x\}}\cap (W \setminus Y)}\ i_{\widetilde{\{x\}}\cap (W \setminus Y)}^{W \setminus Y}
\Big) i_{W\setminus Y}^{Z} \\
& = \sum_{x\to y, x\in W \setminus Y, y \in W \cap Y} r_Y^{\overline{\{y\}}\cap Y}\ i_{\overline{\{y\}}\cap Y}^{\overline{\{y\}}}\ \delta_{\overline{\{y\}}}^{\widetilde{\{x\}}}\ r_{\widetilde{\{x\}}}^{\widetilde{\{x\}}\cap Z}\ i_{\widetilde{\{x\}}\cap Z}^{Z} \\
& = \sum_{(x,y) \in \Lambda_1} \gamma_{x,y}
\end{align*}
where we set 
\[
\Lambda_1 = \{ (x,y) \mid x \to y,\  x \in W \setminus Y,\ y \in W \cap Y\}. 
\]
In a similar way using the facts that $W \setminus Z$ is a closed subset of $Y$ and that $Z$ is an open subset of $Y \cup Z$, we get 
\begin{align*}
r_Y^{W\setminus Z} \delta_{W\setminus Z}^{Z}
& = \sum_{(x,y) \in \Lambda_2} \gamma_{x,y}
\end{align*}
where we set 
\[
\Lambda_2 = \{ (x,y) \mid x \to y,\  x \in W \cap Z,\ y \in W \setminus Z\}. 
\]
If we set 
\begin{align*}
\Lambda_1' 
&= \{ (x,y) \mid x \to y,\  x \in W \cap Y \cap Z,\ y \in W \cap Y\}, \\
\Lambda_2' 
&= \{ (x,y) \mid x \to y,\  x \in W \cap Z,\ y \in W \cap Y \cap Z\} 
\end{align*}
then we have 
\[
\{ (x,y) \mid x \to y,\  x \in W \cap Z,\ y \in W \cap Y\} 
= \Lambda_1 \sqcup \Lambda_1' = \Lambda_2 \sqcup \Lambda_2' 
\]
because $W \subset Y\cup Z$ implies 
$(W \cap Z)\setminus Y = W \setminus Y$ and 
$(W \cap Y) \setminus Z=W \setminus Z$. 
Therefore in order to show the equality 
$\delta_Y^{W\setminus Y} i_{W\setminus Y}^{Z} 
= r_Y^{W\setminus Z} \delta_{W\setminus Z}^{Z}$, 
it suffices to show 
\[
\sum_{(x,y) \in \Lambda_1'} \gamma_{x,y}
= \sum_{(x,y) \in \Lambda_2'} \gamma_{x,y}. 
\]
For each $p \in W \cap Y \cap Z$, we get 
\begin{equation*}
\sum_{y\ot p}  r_{\overline{\{p\}}}^{\overline{\{y\}}}\ \delta_{\overline{\{y\}}}^{\widetilde{\{p\}}}
=\sum_{x\to p} \delta_{\overline{\{p\}}}^{\widetilde{\{x\}}}\ i_{\widetilde{\{x\}}}^{\widetilde{\{p\}}}
\end{equation*}
from the definition of $\Cattakeshi$-modules. 
Multiplying from the left with $r_{Y}^{\overline{\{p\}}\cap Y}\ i_{\overline{\{p\}}\cap Y}^{\overline{\{p\}}}$ and 
from the right with $r_{\widetilde{\{p\}}}^{\widetilde{\{p\}}\cap Z}\ i_{\widetilde{\{p\}}\cap Z}^{Z}$, 
and summing up over $p \in W \cap Y \cap Z$, 
we get 
\begin{align*}
\sum_{p \in W \cap Y \cap Z} & r_{Y}^{\overline{\{p\}}\cap Y}\ i_{\overline{\{p\}}\cap Y}^{\overline{\{p\}}}
\Big(\sum_{y\ot p}  r_{\overline{\{p\}}}^{\overline{\{y\}}}\ \delta_{\overline{\{y\}}}^{\widetilde{\{p\}}} \Big)
r_{\widetilde{\{p\}}}^{\widetilde{\{p\}}\cap Z}\ i_{\widetilde{\{p\}}\cap Z}^{Z} \\
&= \sum_{p \in W \cap Y \cap Z} r_{Y}^{\overline{\{p\}}\cap Y}\ i_{\overline{\{p\}}\cap Y}^{\overline{\{p\}}}
\Big(\sum_{x\to p} \delta_{\overline{\{p\}}}^{\widetilde{\{x\}}}\ i_{\widetilde{\{x\}}}^{\widetilde{\{p\}}}\Big)
r_{\widetilde{\{p\}}}^{\widetilde{\{p\}}\cap Z}\ i_{\widetilde{\{p\}}\cap Z}^{Z}.
\end{align*}
By the relations \ref{it:ii}, \ref{it:rr} and \ref{it:ir}, 
we get
\begin{align} \label{eq:verification_of_rdi}
\sum_{p \in W \cap Y \cap Z} & \sum_{y\ot p}
r_{Y}^{\overline{\{y\}}\cap Y}\ i_{\overline{\{y\}}\cap Y}^{\overline{\{y\}}}\ \delta_{\overline{\{y\}}}^{\widetilde{\{p\}}} 
r_{\widetilde{\{p\}}}^{\widetilde{\{p\}}\cap Z}\ i_{\widetilde{\{p\}}\cap Z}^{Z} \\
&= \sum_{p \in W \cap Y \cap Z} \sum_{x\to p} 
r_{Y}^{\overline{\{p\}}\cap Y}\ i_{\overline{\{p\}}\cap Y}^{\overline{\{p\}}}
\delta_{\overline{\{p\}}}^{\widetilde{\{x\}}}\ 
r_{\widetilde{\{x\}}}^{\widetilde{\{x\}}\cap Z}\ i_{\widetilde{\{x\}}\cap Z}^{Z}.\nonumber
\end{align}
Since $Y$ is locally closed, 
the conditions $p \in W \cap Y \cap Z$, $y\ot p$ and 
$\overline{\{y\}}\cap Y \neq \emptyset$ imply $y \in Y$. 
This further implies $y \in W$ because $W \subset Y \cup W$ is closed. 
Hence the  left-hand side of \eqref{eq:verification_of_rdi} equals 
$\sum_{(x,y) \in \Lambda_1'} \gamma_{x,y}$. 
In a similar way, we can see that 
the  right-hand side of \eqref{eq:verification_of_rdi} equals 
$\sum_{(x,y) \in \Lambda_2'} \gamma_{x,y}$. 
Thus we have proven the relation \ref{it:rdi}, 
and this finishes the verification of all relations 
in Proposition \ref{pro:relations}. 

Hence, $M$ is indeed an $\Catsixterm$-module. To see that $M$ is exact, it remains to show that the sequences $M(C,1)\xrightarrow{\delta_C^U} M(U,0)\xrightarrow{i_U^Y} M(Y,0)$ and $M(Y,1)\xrightarrow{r_Y^C}  M(C,1)\xrightarrow{\delta_C^U} M(U,0)$ are exact for all boundary pairs $(U,C)$ with $Y=U\cup C$.

Fix an element $x\in X$ and consider the commutative diagram
 \[
  \xymatrix{
 M(\{x\},1)\ar[r]^i\ar@{=}[d] & M(\overline{\{x\}},1)\ar[r]^r\ar[d]|\circ & M(\overline{\partial}\{x\},1)\ar[d]|\circ \\
 M(\{x\},1)\ar[r]|\circ & M(\widetilde{\partial}\{x\},0)\ar[r]^i & M(\widetilde{\{x\}},0)
 }
 \]
Using exactness of the upper row and the fact that $N$ was an exact $\Cattakeshi$-module, a diagram chase shows that the bottom row is exact. In a similar way, we see that the sequence
\[
M(\overline{\{x\}},1)\to M(\overline{\partial}\{x\},0)\to M(\{x\},0).
\]
is exact for every $x\in X$.

Next, let $Y\in\LC(X)$ and let $x\in Y$ be a closed point. Then $Y\cap\widetilde{\{x\}}$ is relatively closed in $\widetilde{\{x\}}$ because $Y$ is locally closed. A diagram chase in the commutative diagram
\[
 \xymatrix{
& M(\widetilde{\partial}\{x\} \setminus (Y\cap\widetilde{\partial}\{x\}),0)\ar@{=}[r]\ar[d]^i & M(\widetilde{\{x\}}\setminus (Y\cap\widetilde{\{x\}}),0)\ar[d]^i \\
M(\{x\},1)\ar@{=}[d]\ar[r]|\circ & M(\widetilde{\partial}\{x\},0)\ar@{->>}[d]^r\ar[r]^-i & M(\widetilde{\{x\}},0)\ar@{->>}[d]^r \\
M(\{x\},1)\ar[r]|-\circ &  M(Y\cap\widetilde{\partial}\{x\},0)\ar[r]^-i & M(Y\cap\widetilde{\{x\}},0),
}
\]
whose columns and top row are exact, yields exactness of the bottom row. By a diagram chase in the commutative diagram
\[
 \xymatrix{
M(\{x\},1)\ar@{=}[d]\ar[r]|-\circ &  M(Y\cap\widetilde{\partial}\{x\},0)\ar[r]^-i\ar[d]^i & M(Y\cap\widetilde{\{x\}},0)\ar[d]^i \\
M(\{x\},1)\ar[r]|-\circ &  M(Y\setminus\{x\},0)\ar[r]^-i & M(Y,0)
}
\]
using the exact cosheaf sequence \eqref{eq:cosheaf_condition_on_topology} for the covering $(Y\setminus\{x\},Y\cap\widetilde{\{x\}})$ of $Y$ we obtain exactness of the bottom row. Notice that, using a further diagram chase, it is not hard to deduce the exactness of the cosheaf sequence for a relatively open covering of a locally closed set from the open case.

We have established the exactness of the sequence $M(C,1)\xrightarrow{\delta_C^U} M(U,0)\xrightarrow{i_U^Y} M(Y,0)$ for all boundary pairs $(U,C)$ with $C$ a singleton. Analogously, we find that $M(Y,1)\xrightarrow{r_Y^C}  M(C,1)\xrightarrow{\delta_C^U} M(U,0)$ is exact whenever $U$ is a singleton.

We will proceed by an inductive argument. Let $n\geq 1$ be a natural number and assume that exactness
of the sequence $M(C,1)\xrightarrow{\delta_C^U} M(U,0)\xrightarrow{i_U^Y} M(Y,0)$ is proven for all boundary pairs $(U,C)$ for which $C$ has at most $n$ elements. Let $(U,C)$ be a boundary pair such that $C$ has $n+1$ elements. Write $Y=U\cup C$. Let $p\in C$ be a maximal point and set $U'=U\cup\{p\}$, $C'=C\setminus\{p\}$. Then $(U',C')$ is a boundary pair. A diagram chase in the commutative diagram
\[
 \xymatrix{
M(\{p\},1)\ar@{=}[d]\ar[r]^-i & M(C,1)\ar[d]|-\circ\ar[r]^-r & M(C',1)\ar[d]|-\circ\ar[r]|-\circ &  M(\{p\},0)\ar@{=}[d] \\
M(\{p\},1)\ar[r]|-\circ & M(U,0)\ar[d]^-i\ar[r]^-i & M(U',0)\ar[d]^-i\ar[r]^-r &  M(\{p\},0) \\
& M(Y,0)\ar@{=}[r] & M(Y,0), &
}
\]
whose rows and third column are exact, shows exactness of the second column. Again, exactness of $M(Y,1)\xrightarrow{r_Y^C}  M(C,1)\xrightarrow{\delta_C^U} M(U,0)$ for all boundary pairs follows in a analogous manner. We conclude that $M$ is an exact $\Catsixterm$-module.

Summing up, we have associated an exact \rrzero{} $\Catsixterm$-module with every exact $\Cattakeshi$-module. By a routine argument, this assignment extends uniquely to a functor~$G$ from the category of exact $\Cattakeshi$-modules to the category of exact \rrzero{} $\Catsixterm$-modules. Let $F$ be the restriction of the functor $\Ftakeshi$ to the category of exact \rrzero{} $\Catsixterm$-modules. Then the composition $GF$ is equal to the identity functor on the category of exact $\Cattakeshi$-modules. It remains to show that $FG$ is naturally isomorphic to the identity functor on the category of exact \rrzero{} $\Catsixterm$-modules.

Let $M$ be an exact \rrzero{} $\Catsixterm$-module. We will construct a natural $\Catsixterm$-module isomorphism $\eta_M\colon M\to (FG)(M)$. For $x\in X$ we have $M(\widetilde{\{x\}},0)=(FG)(M)(\widetilde{\{x\}},0)$ and $M(\overline{\{x\}},1)=(FG)(M)(\overline{\{x\}},1)$. Hence we set $\eta_M(\widetilde{\{x\}},0)=\id_{M(\widetilde{\{x\}},0)}$ and $\eta_M(\overline{\{x\}},1)=\id_{M(\overline{\{x\}},1)}$. Using the universal property of kernels and cokernels we obtain natural group homomorphisms $f_Y\colon M(Y,1)\to (FG)(M)(Y,1)$ and $g_Y\colon (FG)(M)(Y,0)\to M(Y,0)$ for every $Y\in\LC(X)$. An application of the Five Lemma shows that these are in fact isomorphisms. We can therefore define $\eta_M(Y,1)=f_Y$ and $\eta_M(Y,0)=(g_Y)^{-1}$.

Finally, we check that this collection of maps constitutes an $\Catsixterm$-module homomorphism, that is, the group homomorphism $\eta_M\colon M\to (FG)(M)$ intertwines the actions of the category $\Catsixterm$ on $M$ and on $(FG)(M)$. By construction this is true for the transformations $(i_{\widetilde{\{x\}}}^{\widetilde{\{y\}}},0)$, $(r_{\overline{\{x\}}}^{\overline{\{y\}}},1)$ and $\delta_{\overline{\{y\}}}^{\widetilde{\{x\}}}$ for all $x,y\in X$ with $x\to y$. By Lemma \ref{lem:sheaf_equivalence} and Lemma \ref{lem:cosheaf_equivalence} it is also true for the transformation $(i_U^V,0)$ for all open subset $U,V$ of $X$ with $U\subset V$ and for $(r_C^D,1)$ for all closed subsets $C,D$ of $X$ with $D\subset C$.

Let $V\subset X$ be open and let $Y\subset V$ be relatively closed. Since $(r_V^Y,0)$ was defined as a natural projection onto a cokernel, our assertion holds for this transformation as well. Consequently, by \eqref{eq:r_relation} the assertion also follows for the transformation $(i_U^Y,0)$ for $Y\in\LC(X)$ and $U\subset Y$ relatively open. Finally \eqref{eq:def_of_arbitrary_r} implies the assertion for the transformation $r_Y^C$ with $Y\in\LC(X)$ and $C\subset Y$ relatively closed. We have shown that $\eta$ intertwines the actions of all even transformations on the $0$-parts of $M$ and $(FG)(M)$. By analogous arguments the same follows for the actions of all even transformations on the $1$-parts of $M$ and $(FG)(M)$.

Our last step is to consider the action of a boundary transformation $\delta_C^U$ for a boundary pair $(U,C)$. Since $M$ and $(FG)(M)$ are \rrzero{} the $0$\nb-to\nb-$1$ component of $\delta_C^U$ acts trivially on both modules. We have already seen that the assertion is true for the $1$\nb-to\nb-$0$ component of $\delta_C^U$ in the specific case that $(U,C)=(\widetilde{\{x\}},\overline{\{y\}})$ with $x\to y$. The general case then follows from  \eqref{eq:def_of_delta} since $X$ is an EBP space. 
\end{proof}

\section{Reduced filtered \texorpdfstring{$\K$}{K}-theory} \label{gunnar}

Let $X$ be an arbitrary finite $T_0$-space. We recall some definitions and facts from~\cite{range_result}.
In~\cite{restorff}, Gunnar Restorff introduced reduced filtered $\K$\nb-the\-ory $\FKgunnar$ and showed that it classifies purely infinite \CK{}s up to stable isomorphism. In~\cite{range_result}, the range of reduced filtered $\K$\nb-the\-ory is established with respect to purely infinite \CK{}s.

\begin{definition}[\cite{range_result}*{Definition~3.1}]
Let $\Catgunnar$ denote the universal pre-additive category generated by objects $\gsi x, \gobd x, \gosi x $ for all $x\in X$ and morphisms $\delta_{\gsi x}^{\gobd x}$ and $i_{\gobd x}^{\gosi x}$ for all $x\in X$, and $i_{\gosi y}^{\gobd x}$ when $y\to x$, subject to the relations
\begin{equation}
  \label{eq:gunnarrelation1}
\delta_{\gsi x}^{\gobd x} i_{\gobd x}^{\gosi x} =0
\end{equation}
\begin{equation}
  \label{eq:gunnarrelation2}
 i_p i_{\gosi{y(p)}}^{\gobd x} =  i_q i_{\gosi{y(q)}}^{\gobd x}
\end{equation}
for all $x\in X$, all $y\in X$ satisfying $y> x$, and all paths $p,q\in\Path(y,x)$, where for a path $p=(z_k)_{k=1}^n$ in $\Path(y,x)$, we define $y(p)=z_2$, and
\[
i_p = i_{\gosi{z_n}}^{\gobd{z_{n-1}}} i_{\gobd{z_{n-1}}}^{\gosi{z_{n-2}}}\cdots i_{\gosi{z_3}}^{\gobd{z_2}} i_{\gobd{z_2}}^{\gosi{z_2}} .
\]
\end{definition}

\begin{definition}
It is easy to see that the relations in $\Catsixterm$ corresponding to \eqref{eq:gunnarrelation1} and \eqref{eq:gunnarrelation2} hold. We can thus define an additive functor $\Catgunnar\to\Catsixterm$ by $\gsi x\mapsto (\si x,1)$, $\gobd x\mapsto (\ob \{x\},0)$ and $\gosi x\mapsto (\osi x,0)$, and in the obvious way on morphisms. Let $\Fgunnar\colon\Mod(\Catsixterm)\to\Mod(\Catgunnar)$ denote the induced functor. Define \emph{reduced filtered $\K$\nb-the\-ory}, $\FKgunnar$ as the composition of $\FKsixterm$ with $\Fgunnar$.
\end{definition}

An equivalent definition of the functor~$\FKgunnar$ is given in \cite{range_result}*{Definition~3.4}.

\begin{definition}[\cite{range_result}*{Definition~3.6}]
An $\Catgunnar$-module $M$ is called \emph{exact} if the sequences
\begin{equation}
  \label{eq:Catgunnarrel1}
M(\gsi x) \xrightarrow{\delta} M(\gobd x) \xrightarrow{i} M(\gosi x)
\end{equation}
\begin{equation}
  \label{eq:Catgunnarrel2}
\bigoplus_{(p,q)\in \DoublePaths(x)} M(\gosi{s(p,q)}) \xrightarrow{(i_p-i_q)} \bigoplus_{y\to x} M(\gosi y) \xrightarrow{(i_{\gosi y}^{\gobd x})} M(\gobd x) \longrightarrow 0
\end{equation}
are exact for all $x\in X$, where $\DoublePaths(x)$ denotes the set of pairs of distinct paths $(p,q)$ to $x$ and from some common element which is denoted $s(p,q)$.
\end{definition}

The following lemma is a generalization of~\cite{range_result}*{Lemma~3.9}.  We omit the proof as the same technique applies here.

\begin{lemma} \label{lem:exactseq}
Let $M$ be an exact \rrzero{} $\Catsixterm$-module.
Let $Y$ be an open subset of $X$ and let $(U_i)_{i\in I}$ be an open covering of $Y$. Then the following sequence is exact:
\[
\bigoplus_{i,j\in I} M(U_i\cap U_j,0)\xrightarrow{(i_{U_i\cap U_j}^{U_i}-i_{U_i\cap U_j}^{U_j})}
\bigoplus_{i\in I} M(U_i,0)\xrightarrow{(i_{U_i}^Y)}
M(Y,0) \longrightarrow 0.
\]
\end{lemma}

\begin{corollary} \label{cor:STexactR}
Let $M$ be an exact \rrzero{} $\Catsixterm$-module and set $N=\Fgunnar(M)$. Then $N$ is an exact $\Catgunnar$-module.
\end{corollary}

\begin{remark}
 If $X$ is a unique path space, then the set $\DoublePaths(x)$ is empty for every $x\in X$. Hence, for an exact $\Catgunnar$-module $M$, the map $(i_{\gosi y}^{\gobd x})\colon\displaystyle\bigoplus_{y\to x} M(\gosi y) \to M(\gobd x)$ is an isomorphism. In this sense, the groups $M(\gobd x)$ are redundant for an exact $\Catgunnar$-module in case $X$ is a unique path space.
\end{remark}

By combining the following Proposition~\ref{prop:range} and Theorem~\ref{thm:range}, one may obtain a complete description of the range of reduced filtered $\K$\nb-the\-ory for purely infinite graph \csa{}s and \CK{}s.

\begin{proposition}[{\cite{range_result}*{Proposition~4.7}}] \label{prop:range}
Let $A$ be a purely infinite graph \csa{} over $X$.  Then $\FKgunnar(A)$ is an exact $\Catgunnar$-module, and $\FK_{\si x}^1(A)$ is free for all $x\in X$.

If $A$ is a purely infinite \CK{} over $X$, then furthermore $\K_1(A(x))$ and $\K_0(A(\osi x))$ are finitely generated, and the rank of $\K_1(A(x))$ coincides with the rank of the cokernel of the map $i\colon \K_0(A(\obd x)\to \K_0(A(\osi x))$, for all $x\in X$.
\end{proposition}

\begin{theorem}[\cite{range_result}*{Theorem~4.8}] \label{thm:range}
Let $M$ be an exact $\Catgunnar$-module  with $M(\gsi x)$ free for all $x\in X$.
Then there exists a countable graph $E$ satisfying that all vertices in $E$ are regular and support at least two cycles, that $C^*(E)$ is tight over~$X$ and that $\FKgunnar\bigl(C^*(E)\bigr)$ is isomorphic to~$M$.
By construction $C^*(E)$ is purely infinite.

The graph $E$ can be chosen to be finite if \textup{(}and only if\textup{)} $M(\gsi x)$ and $M(\gosi x)$ are finitely generated, and the rank of $M(\gsi x)$ coincides with the rank of the cokernel of $i\colon M(\gobd x)\to M(\gosi x)$, for all $x\in X$.  
If $E$ is chosen finite, then by construction $C^*(E)$ is a Cuntz--Krie\-ger algebra.
\end{theorem}

In Corollary~\ref{cor:range}, we combine this range-of-invariant theorem with the isomorphism lifting result from the next section.

\section{An intermediate invariant} \label{sec:intermediate}

In this section, we define one more invariant, which, in a sense, can be thought of as a union or join of reduced filtered $\K$\nb-the\-ory $\FKgunnar$ and filtered $\K$\nb-the\-ory restricted to canonical base $\FKtakeshi$. It functions as an intermediate invariant towards concrete filtered $\K$\nb-the\-ory $\FKsixterm$.

Let $X$ be a unique path space.

\begin{definition}
Let $\Catgunnartakeshi$ denote the universal pre-additive category generated by objects $\gsi x$, $\tcsi x$, $\gosi x$ for all $x\in X$ and morphisms $i_{\gsi x}^{\tcsi x}$ for all $x\in X$ and $r_{\tcsi x}^{\tcsi y}$, $\delta_{\tcsi y}^{\gosi x}$ and $i_{\gosi x}^{\gosi y}$ when $x\to y$, subject to the relations
\begin{equation}
\sum_{x\to y}  r_{\tcsi x}^{\tcsi y} \delta_{\tcsi y}^{\gosi x} = \sum_{z\to x} \delta_{\tcsi x}^{\gosi z} i_{\gosi z}^{\gosi x}
\end{equation}
for all $x\in X$ and
\begin{equation}
  \label{eq:RKrelation2}
i_{\gsi x}^{\tcsi x}r_{\tcsi x}^{\tcsi y}=0
\end{equation}
when $x\to y$.

As before, there is a canonical additive functor $\Catgunnartakeshi\to\Catsixterm$, inducing a functor $\Fgunnartakeshi\colon\Mod(\Catsixterm)\to\Mod(\Catgunnartakeshi)$.
Define $\FKgunnartakeshi$ as the composition of $\FKsixterm$ with $\Fgunnartakeshi$.

The category $\Cattakeshi$ embeds into the category $\Catgunnartakeshi$, and a forgetful functor $\Mod(\Catgunnartakeshi)\to\Mod(\Cattakeshi)$ is induced.
We define an additive functor $\FGTgunnar\colon\Mod(\Catgunnartakeshi)\to\Mod(\Catgunnar)$ by
\[
M(\gobd x) = \bigoplus_{y\to x} M(\gosi y)
\]
and $\delta_{\gsi x}^{\gobd x}=(i_{\gsi x}^{\tcsi x} \delta_{\tcsi x}^{\gosi y})$ and otherwise in the obvious way.
\end{definition}

\begin{definition}
A $\Catgunnartakeshi$-module $M$ is called \emph{exact} if the sequences
\begin{equation}
M(\tcsi x) \xrightarrow{\begin{pmatrix} r_{\tcsi x}^{\tcsi y} & -\delta_{\tcsi x}^{\gosi z} \end{pmatrix}} \bigoplus_{x\to y} M(\tcsi y) \oplus \bigoplus_{z\to x} M(\gosi z) \xrightarrow{\begin{pmatrix} \delta_{\tcsi y}^{\gosi x} \\ i_{\gosi z}^{\gosi x} \end{pmatrix}} M(\gosi x)
\end{equation}
\begin{equation}
  \label{eq:RKexatness2}
0 \to M(\gsi x) \xrightarrow{i_{\gsi x}^{\tcsi x}} M(\tcsi x) \xrightarrow{(r_{\tcsi x}^{\tcsi y})} \bigoplus_{x\to y} M(\tcsi y)
\end{equation}
are exact for all $x\in X$.
\end{definition}

\begin{lemma}
 Let $M$ be an exact \rrzero{} $\Catsixterm$-module. Then $\Fgunnartakeshi(M)$ is an exact $\Catgunnartakeshi$-module.
\end{lemma}

\begin{proof}
The proof is similar to the proof of Lemma \ref{lem:takeshi_exact}. 
\end{proof}

\begin{theorem} \label{thm:gunnartotakeshi}
Assume that $X$ is a unique path space.  Let $M$ and $N$ be exact $\Catgunnartakeshi$-modules with $M(\gsi x)$ and $N(\gsi x)$ free for all non-open points $x\in X$, and let $\phi\colon\FGTgunnar(M)\to\FGTgunnar(N)$ be an $\Catgunnar$-module homomorphism. Then there exists a \textup{(}not necessarily unique\textup{)} $\Catgunnartakeshi$-module homomorphism $\Phi\colon M\to N$ such that $\FGTgunnar(\Phi)=\phi$. If~$\phi$ is an isomorphism then, by construction, so is~$\Phi$.
\end{theorem}
\begin{proof}
For $x\in X$, we define $\Phi_{x_1}=\phi_{\gsi x}$ and $\Phi_{{\gosi x}} = \phi_{{\gosi x}}$. In the following, we will define $\Phi_{{\tcsi x}}$ by induction on the partial order of $X$ in a way such that the relations
\begin{equation}
  \label{eq:proofgunnartotakeshi0a}
r_{\tcsi x}^{\tcsi y} \Phi_{\tcsi y}=\Phi_{\tcsi x} r_{\tcsi x}^{\tcsi y},
\end{equation}
\begin{equation}
  \label{eq:proofgunnartotakeshi0b}
\delta_{\tcsi x}^{\gosi z} \Phi_{\gosi z}=\Phi_{\tcsi x} \delta_{\tcsi x}^{\gosi z}
\end{equation}
\begin{equation}
  \label{eq:proofgunnartotakeshi0c}
i_{\gsi x}^{\tcsi x} \Phi_{\tcsi x}=\Phi_{\gsi x} i_{\gsi x}^{\tcsi x}
\end{equation}
hold for all $y$ with $x\to y$ and all $z$ with $z\to x$. For closed points $x\in X$, we set
\[
 \Phi_{{\tcsi x}}=i_{\gsi x}^{{\tcsi x}}\phi_{x_1}\left(i_{x_1}^{{\tcsi x}}\right)^{-1}.
\]
Here we have used that, by exactness of \eqref{eq:RKexatness2}, $i_{x_1}^{{\tcsi x}}$ is invertible as there is no $y$ with $x\to y$. While the condition \eqref{eq:proofgunnartotakeshi0a} is empty, \eqref{eq:proofgunnartotakeshi0b} is guaranteed by $\phi$ being an $\Catgunnar$-module homomorphism,
and  \eqref{eq:proofgunnartotakeshi0c}  holds by construction.

Now fix an element $w\in X$ and assume that $\Phi_{\tcsi x}$ is defined for all $x<w$ in a way such that \eqref{eq:proofgunnartotakeshi0a} and \eqref{eq:proofgunnartotakeshi0b} hold. Using the exact sequence \eqref{eq:RKexatness2} and the freeness of $\displaystyle\bigoplus_{w\to x} M(\tcsi w)$, we can choose a free subgroup $V\subseteq M(\tcsi w)$ such that $M(\tcsi w)$ decomposes as an inner direct sum
\[
M(\tcsi w)=V\oplus M(\gsi w)\cdot i_{w_1}^{{\tcsi w}}.
\]
We will define $\Phi_{\tcsi w}$ by specifying the two restrictions $\Phi_{\tcsi w}|_V$ and $\Phi_{\tcsi w}|_{M(\gsi w)\cdot i_{w_1}^{{\tcsi w}}}$. Consider the diagram
\begin{equation}
  \label{eq:proofgunnartotakeshi1}
\begin{split}
\xymatrix@C+25pt{
V\ar@{ >->}[r]\ar@{..>}[rd] & M(\tcsi x) \ar[r]^-{\scriptstyle{\begin{pmatrix} r_{\tcsi x}^{\tcsi y}, -\delta_{\tcsi x}^{\gosi z} \end{pmatrix}}} &
\displaystyle\bigoplus_{x\to y} M(\tcsi y) \oplus \displaystyle\bigoplus_{z\to x} M(\gosi z) \ar[r]^-{\scriptstyle{\begin{pmatrix} \delta_{\tcsi y}^{\gosi x} \\ i_{\gosi z}^{\gosi x} \end{pmatrix}}} \ar[d]^{\bigl(\left(\Phi_{\tcsi y}\right),\left( \Phi_{\gosi z}\right)\bigr)}
& M(\gosi x) \ar[d]^{\Phi_{\gosi x}} \\
& N(\tcsi x) \ar[r]_-{\scriptstyle{\begin{pmatrix} r_{\tcsi x}^{\tcsi y}, -\delta_{\tcsi x}^{\gosi z} \end{pmatrix}}} &
\displaystyle\bigoplus_{x\to y} N(\tcsi y) \oplus \displaystyle\bigoplus_{z\to x} N(\gosi z) \ar[r]_-{\scriptstyle{\begin{pmatrix} \delta_{\tcsi y}^{\gosi x} \\ i_{\gosi z}^{\gosi x} \end{pmatrix}}}
& N(\gosi x)
}
\end{split}
\end{equation}
By assumption, the bottom row of this diagram is exact, the top row is exact in $\displaystyle\bigoplus_{x\to y} M(\tcsi y) \oplus \displaystyle\bigoplus_{z\to x} M(\gosi z)$, and the right-hand square commutes. We can therefore choose a homomorphism $\Phi_{\tcsi x}|_V\colon V\to N(\tcsi x)$ such that the left-hand pentagon commutes.

By exactness of \eqref{eq:RKexatness2}, $i_{x_1}^{{\tcsi x}}$ is injective. Its corestriction onto its image $M(\gsi x)\cdot i_{x_1}^{{\tcsi x}}$ is thus an isomorphism. We may therefore define the restriction $\Phi_{\tcsi x}|_{M(\gsi x)\cdot i_{x_1}^{{\tcsi x}}}$ in the unique way that makes the following diagram commute:
\begin{equation}
  \label{eq:proofgunnartotakeshi2}
\begin{split}
\xymatrix{
M(\gsi x)\ar[r]^-{i_{x_1}^{{\tcsi x}}}\ar[d]^{\phi_{\gsi x}} & M(\gsi x)\cdot i_{x_1}^{{\tcsi x}}\ar[d]^{\Phi_{\tcsi x}|_{M(\gsi x)\cdot i_{x_1}^{{\tcsi x}}}} \\
N(\gsi x)\ar[r]^-{i_{x_1}^{{\tcsi x}}} & N(\gsi x)\cdot i_{x_1}^{{\tcsi x}}
}
\end{split}
\end{equation}
We have to check that $\Phi_{\tcsi w}=(\Phi_{\tcsi w}|_V,\Phi_{\tcsi w}|_{M(\gsi w)\cdot i_{w_1}^{{\tcsi w}}})$ fulfills \eqref{eq:proofgunnartotakeshi0a} and \eqref{eq:proofgunnartotakeshi0b} (with $x$ replaced with $w$). This is true on $V$ because of the commutativity of the left-hand side of \eqref{eq:proofgunnartotakeshi1}. It is also true on the second summand: by \eqref{eq:RKrelation2}, both sides of \eqref{eq:proofgunnartotakeshi0a} vanish on this subgroup; \eqref{eq:proofgunnartotakeshi0b} follows again from $\phi$ being an $\Catgunnar$-module homomorphism; and  \eqref{eq:proofgunnartotakeshi0c} holds by construction. This completes the induction step.

The claim, that $\Phi$ is an isomorphism whenever $\phi$ is, follows from a repeated application of the Five Lemma.
\end{proof}

\begin{corollary} \label{cor:gunnartosixterm}
Assume that $X$ is an EBP space.  Let $M$ and $N$ be exact, \rrzero{} $\Catsixterm$-modules with $M(\si x,1)$ and $N(\si x,1)$ free for all non-open points $x\in X$, and let $\phi\colon\Fgunnar(M)\to\Fgunnar(N)$ be an $\Catgunnar$-module homomorphism. Then there exists a \textup{(}not necessarily unique\textup{)} $\Catsixterm$-module homomorphism $\Phi\colon M\to N$ satisfying $\Fgunnar(\Phi)=\phi$. If~$\phi$ is an isomorphism then, by construction, so is~$\Phi$.
\end{corollary}

\begin{proof}
 Combine Theorems \ref{thm:gunnartotakeshi} and \ref{thm:sixterm_to_takeshi}.
\end{proof}

\begin{corollary} \label{cor:gunnartosixterm_cstar}
Let $A$ and $B$ be \csa s of real rank zero over an EBP space~$X$, and assume that $\K_1\bigl(A(x)\bigr)$ and $\K_1\bigl(B(x)\bigr)$ are free abelian groups for all non-open points $x\in X$. Then for any homomorphism $\phi\colon\FKgunnar(A)\to\FKgunnar(B)$, there exist a \textup{(}not necessarily unique\textup{)} homomorphism $\Phi\colon\FKsixterm(A)\to\FKsixterm(B)$ for which $\Fgunnar(\Phi)=\phi$. If~$\phi$ is an isomorphism then, by construction, so is~$\Phi$.
\end{corollary}

\begin{corollary} \label{cor:range}
Let $A$ be a \csa{} over~$X$ with real rank zero, and assume that $\K_1\bigl(A(x)\bigr)$ is free for all $x\in X$.
Then there exists a purely infinite graph \csa{} $C^*(E)$ that is tight over~$X$ and satisfies $\FKgunnar\bigl(C^*(E)\bigr)\cong\FKgunnar(A)$.
If $X$ is an EBP space, then automatically $\FKsixterm\bigl(C^*(E)\bigr)\cong\FKsixterm(A)$.

If furthermore for all $x\in X$, the group $\K_*\bigl(A(x)\bigr)$ is finitely generated and $\rank\K_1\bigl(A(x)\bigr)=\rank\K_0\bigl(A(x)\bigr)$, then $C^*(E)$ can be chosen to be a purely infinite Cuntz--Krie\-ger algebra.
\end{corollary}
\begin{proof}
Combine Theorem~\ref{thm:range} with Corollary~\ref{cor:STexactR} and Corollary~\ref{cor:gunnartosixterm_cstar}.
\end{proof}

\subsection{The particular case of the four-point space \texorpdfstring{$\mathcal D$}{D}}
Consider the space $\mathcal D=\{1,2,3,4 \}$ defined by $4\to 3,4\to 2,3\to1,2\to1$.  The space $\mathcal D$ is not a unique path space. The second-named author showed in~\cite{bentmann} that there exists a finite refinement $\FK'$ of filtered $\K$\nb-the\-ory $\FK$ given by adding a \csa{} $R_{1\setminus 4}$ over~$\mathcal D$ to the collection $(R_Y)_{Y\in\LC(\mathcal D)^*}$ of representing objects, creating a larger category $\NT'$.
By~\cite{bentmann}*{Theorem 6.2.14}, isomorphisms on the refined filtered $\K$\nb-the\-ory $\FK'$ lift to $\KK(\mathcal D)$-equivalences, and thereby (using~\cite{kirchberg}) to $\mathcal D$-equivariant \Star{}iso\-mor\-phisms, for stable Kirchberg $\mathcal D$-algebras with all simple subquotients in the bootstrap class.
However, there exist two non-isomorphic stable Kirchberg $\mathcal D$-algebras $A$ and~$B$ with real rank zero and simple subquotients in the bootstrap class such that $\FK(A)\cong\FK(B)$, see~\cites{arr,bentmann}.

\begin{proposition} \label{prop:diamond}
Let $A$ and $B$ be \csa s over $\mathcal D$, assume that $A$ and $B$ have real rank zero, and assume that $\K_1\bigl(A(x)\bigr)$ and $\K_1\bigl(B(x)\bigr)$ are free abelian groups for all $x\in\{1,2,3\}$.
Then any homomorphism $\phi\colon\FKgunnar(A)\to\FKgunnar(B)$ extends \textup{(}non-uniquely\textup{)} to a homomorphism $\Phi\colon\FK'(A)\to\FK'(B)$.
If $\phi$ is an isomorphism, then $\Phi$ is by construction an isomorphism.
\end{proposition}

\begin{proof}
By Section~6.2.5 of~\cite{bentmann}, the refined filtered $\K$\nb-the\-ory $\FK'$ consists of the following groups and maps:
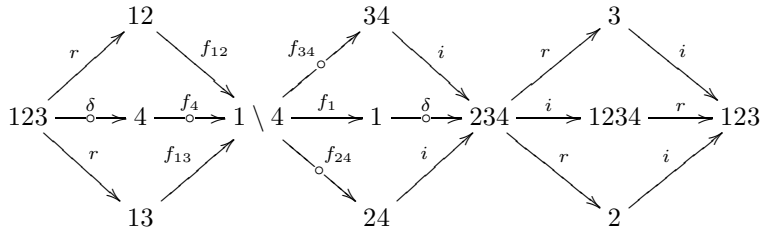
\[
\xymatrix{
& 12\ar[dr]^-{f_{12}} & & 34\ar[dr]^-i & & 3\ar[dr]^-i & \\
123 \ar[ur]^-r \ar[dr]^-r \ar[r]|-\circ^-\delta & 4 \ar[r]|-\circ^-{f_{4}} & 1\setminus 4 \ar[ur]|-\circ^-{f_{34}} \ar[dr]|-\circ^-{f_{24}} \ar[r]^-{f_{1}} & 1 \ar[r]|-\circ^-\delta & 234\ar[ur]^-r\ar[dr]^-r  \ar[r]^-i & 1234 \ar[r]^-r & 123 \\
& 13\ar[ur]^-{f_{13}} & & 24\ar[ur]^-i & & 2\ar[ur]^-i & 
}
\]
The proof of~\cite{arr}*{Lemma 3.4} applies to the space $\mathcal D$, hence the two triangles
\[ \xymatrix@!0@C=60pt@R=40pt{
\FK_{234}(A) \ar[rr]^-{r_{234}^3i_3^{123}\delta_{123}^4f_4} && \FK_{1\setminus 4}(A) \ar[dl]|-\circ^(.4){(f^{34},f^1,f^{24})} \\
&  \FK_{34}(A)\oplus\FK_{1}(A)[1]\oplus\FK_{23}(A) \ar[ul]^(.6){i_{34}^{234}+\delta_1^{234}+i_{24}^{234}} & \\
\FK_{1\setminus 4}(A) \ar[rr]|-\circ^-{f^1\delta_1^{234}r_{234}^3i_3^{123}} && \FK_{123}(A) \ar[dl]^(.4){(r_{123}^{12},\delta_{123}^4,r_{123}^{13})} \\
& \FK_{12}(A)\oplus\FK_{4}(A)[1]\oplus\FK_{13}(A) \ar[ul]^(.6){f_{12}+f_4+f_{13}}
} \]
are exact.
Since $A$ is of real rank zero, the maps
\[ 
\FK_{123}^0(A)\xrightarrow{\delta_{123}^4}\FK_4^1(A),  \quad
\FK_{12}^0(A)\xrightarrow{\delta_{12}^{34}}\FK_{34}^1(A), 
\] \[
\FK_{13}^0(A)\xrightarrow{\delta_{13}^{24}}\FK_{24}^1(A), \quad
\FK_{1}^0(A)\xrightarrow{\delta_{1}^{234}}\FK_{234}^1(A)
\]
vanish by Proposition~4 of~\cite{linrordam}.
So for \csa s over $\mathcal D$ of real rank zero, the invariant $\FK'$ with the group $1\setminus 4_0$ and its related maps omitted, consists of the following groups and maps:
\[ {\scriptscriptstyle{ \xymatrix@!0@C=28pt@R=28pt{
& 24_1\ar[dr] & & 2_1\ar[dr] & & 12_1\ar[dr] & & 34_0\ar[dr] & & 3_0\ar[dr] & & 13_0\ar[dr] & \\
4_1\ar[ur]\ar[dr] & & 234_1\ar[ur]\ar[r]\ar[dr] & 1234_1\ar[r] & 123_1\ar[ur]\ar[r]\ar[dr] & 4_0\ar[r] & 1\setminus 4_1\ar[ur]\ar[r]\ar[dr] & 1_1\ar[r] & 234_0\ar[ur]\ar[r]\ar[dr] & 1234_0\ar[r] & 123_0\ar[ur]\ar[dr] && 1_0 \\
& 34_1\ar[ur] & & 3_1\ar[ur] & & 13_1\ar[ur] & & 24_0\ar[ur] & & 2_0\ar[ur] & & 12_0\ar[ur] & 
}}}
\]
The reduced filtered $\K$\nb-the\-ory $\FKgunnar$ consists of the sequences $3_1\to 4_0\to 34_0$, $2_1\to 4_0\to 24_0$, $1_1\to 234_0\to 1234_0$ together with the maps $34_0\to 234_0$ and $24_0\to 234_0$ and the group $4_1$.

We will now construct $\Phi=(\Phi^*_Y)_{Y\in\LC(\mathcal D)^*\cup\{1\setminus 4\}}$ from $\phi$.
Define $\Phi_{\si x}^1=\phi_{\gsi x}$, $\Phi_{\ob x}^0=\phi_{\gobd x}$, and $\Phi_{\osi x}^0=\phi_{\gosi x}$ for all $x\in\mathcal D$.
For $Y\in\{3,2,123,13,12,1\}$, the maps $\Phi_Y^0$ are constructed as the induced maps on cokernels, as in the proof of Theorem~\ref{thm:sixterm_to_takeshi}.

Since $\FK_1^1(A)$ is free and the sequence
\[ 0\longrightarrow \FK_4^0(A) \xrightarrow{f_4} \FK_{1\setminus 4}^1(A) \xrightarrow{f^1} \FK_1^1(A) \longrightarrow 0 \]
is exact, we can find a free subgroup $V_{1\setminus 4}$ of $\FK_{1\setminus 4}^1(A)$ for which $\im f_4\oplus V_{1\setminus 4}=\FK_{1\setminus 4}^1(A)$.
Consider the commuting diagram
\[ \xymatrix@C+30pt{
V_{1\setminus 4} \ar[r]^-{(f^{34},f^1,f^{24})} \ar@{..>}[d] & \FK_{34}^0(A)\oplus\FK_1^1(A)\oplus\FK_{24}^0(A) \ar[r]^-{i_{34}^{234}+\delta_1^{234}+i_{24}^{234}}\ar[d]^-{\Phi_{34}^0\oplus\Phi_1^1\oplus\Phi_{24}^0} & \FK_{234}^0(A)\ar[d]^-{\Phi_{234}^0} \\
\FK_{1\setminus 4}(B) \ar[r]^-{(f^{34},f^1,f^{24})} & \FK_{34}^0(B)\oplus\FK_1^1(B)\oplus\FK_{24}^0(B) \ar[r]^-{i_{34}^{234}+\delta_1^{234}+i_{24}^{234}} & \FK_{234}^0(B) .
} \]
Since the bottom row is exact and the top row is a complex, and due to freeness of $V_{1\setminus 4}$, we may choose a map $\psi\colon V_{1\setminus 4}\to\FK_{1\setminus 4}^1(B)$ that makes the left square of the diagram commute.
Define $\Phi_{1\setminus 4}^1$ on $\im f_4\oplus V_{1\setminus 4}$ as $\Phi_4^0 + \psi$.  By construction,
\[
\Phi_{1\setminus 4}^1 f_4 = f_4 \Phi_4^0, \quad
f^1 \Phi_{1\setminus 4}^1 = \Phi_1^1 f^1, \quad
f^{34} \Phi_{1\setminus 4}^1= \Phi_{34}^0 f^{34}, \quad
f^{24} \Phi_{1\setminus 4}^1 = \Phi_{24}^0 f^{24},
\]
and by the Five Lemma, the homomorphism $\Phi_{1\setminus 4}^1$ is an isomorphism if $\phi$ is an isomorphism.

Similarly, to construct $\Phi_{12}^1$, use exactness of the sequence
\[ 0\longrightarrow \FK_2^1(A) \xrightarrow{i_2^{12}} \FK_{12}(A) \xrightarrow{r_{12}^1} \FK_1^1(A) \xrightarrow{\delta_1^2} \FK_2^0(A) \]
and freeness of $\FK_1^1(A)$ to choose a free subgroup $V_{12}$ of $\FK_{12}^1(A)$ for which $\im i_2^{12}\oplus V_{12}=\FK_{12}^1(A)$.
Consider the commuting diagram
\[ \xymatrix{
V_{12} \ar@{..>}[d] \ar[r]^-{f_{12}} & \FK_{1\setminus 4}^1(A) \ar[d]^-{\Phi_{1\setminus 4}^1}\ar[r]^-{f^{24}} & \FK_{24}^0(A) \ar[d]^-{\Phi_{24}^0} \\
\FK_{12}^1(B) \ar[r]^-{f_{12}} & \FK_{1\setminus 4}^1(B) \ar[r]^-{f^{24}} & \FK_{24}^0(B) .
} \]
Using exactness of the bottom row and that the top row is a complex,
the map $\Phi_{12}^1$ can be constructed so that
\[
\Phi_{12}^1 i_2^{12} = i_2^{12} \Phi_2^1, \quad
f_{12} \Phi_{12}^1 = \Phi_{1\setminus 4}^1 f_{12} .
\]
Again due to the Five Lemma, $\Phi_{12}^1$ is an isomorphism if $\phi$ is.
The maps $\Phi_{13}^{1}$, $\Phi_{123}^{1}$, $\Phi_{1234}^{1}$, $\Phi_{234}^{1}$, $\Phi_{34}^{1}$, and $\Phi_{24}^{1}$ are constructed similarly and in the specified order.

Finally, the group $\FK_{1\setminus 4}^0(A)$ is naturally isomorphic to
\begin{align*}
&\coker\bigl(\FK_{123}^0(A)\xrightarrow{(r_{123}^{12},\delta_{123}^4,r_{123}^{13})}\FK_{12}^0(A)\oplus\FK_4^1(A)\oplus\FK_{13}^1(A)\bigr) \\ &= \FK_4^1(A)\oplus\coker\bigl(\FK_{123}^0(A)\xrightarrow{(r_{123}^{12},r_{123}^{13})}\FK_{12}^0(A)\oplus\FK_{13}^1(A)\bigr)
\end{align*}
whose second summand, due to real rank zero, is naturally isomorphic to $\FK_1^0(A)$
Therefore, by defining $\Phi_{1\setminus 4}^0$ as the map induced by $\Phi_4^1\oplus\Phi_1^0$, $\Phi$ becomes a $\NT'$-morphism.
\end{proof}

\begin{corollary}
Let $A$ and $B$ be \csa s over $\mathcal D$.  Assume that $A$ and $B$ have real rank zero, that $\K_1\bigl(A(x)\bigr)$ and $\K_1\bigl(B(x)\bigr)$ are free abelian groups for all $x\in\{1,2,3\}$, and that $A$ and $B$ are in the bootstrap class of Meyer--Nest.
Then any isomorphism $\FKgunnar(A)\to\FKgunnar(B)$ lifts to a $\KK(\mathcal D)$-equivalence.
\end{corollary}

\begin{proof}
Combine Proposition~\ref{prop:diamond} with~\cite{bentmann}*{Theorem 6.2.14}.
\end{proof}

\begin{corollary}
Let $A$ and $B$ be stable Kirchberg $\mathcal D$-algebras of real rank zero, assume that $\K_1\bigl(A(x)\bigr)$ and $\K_1\bigl(B(x)\bigr)$ are free abelian groups for all $x\in\{1,2,3\}$, and assume that $A(x)$ and $B(x)$ are in the bootstrap class for all $x\in\mathcal D$.
Then any isomorphism $\FKgunnar(A)\to\FKgunnar(B)$ lifts to a $\mathcal D$-equivariant \Star{}isomorphism $A\to B$.
\end{corollary}

\begin{proof}
Combine Proposition~\ref{prop:diamond} with~\cite{bentmann}*{Theorem 6.2.15}.
\end{proof}

\section{Unital filtered \texorpdfstring{$\K$}{K}-theory} \label{sec:unital}

In~\cite{rr}*{2.1}, Gunnar Restorff and Efren Ruiz showed that if a functor~$F$ (that factors through the functor~$\K_0$) strongly classifies a certain type of class of \csa s up to stable isomorphism, then the functor $A\mapsto\bigl(F(A),[1_A]\in\K_0(A)\bigr)$ classifies unital, properly infinite \csa s in the class up to isomorphism. A version with slightly generalized assumptions of this so-called \emph{meta-theorem} may be found in~\cite{err_meta} as Theorem~3.3. With these generalized assumptions, the theorem applies to filtered $\K$\nb-the\-ory $\FK$ over accordion spaces~$X$ with respect to Kirchberg $X$\nb-al\-ge\-bras with simple subquotients in the bootstrap class.

Let $X$ be an arbitrary finite $T_0$-space. For $x,x'\in X$, we let $\inf(x,x')$ denote the set $\{ y\in X \mid y\to x, y\to x'\}$.

\begin{definition}
The category $\Mod(\Catsixterm)^\pt$ of \emph{pointed $\Catsixterm$-modules} is defined to have objects $(M,m)$ where $M$ is a $\Catsixterm$-module and $m\in M(X,0)$, and morphisms $\phi\colon(M,m)\to(N,n)$ that are $\Catsixterm$-morphisms with $\phi(m)=n$.

The category $\Mod(\Cattakeshi)^\pt$ of \emph{pointed $\Cattakeshi$-modules} is defined similarly with objects $(M,m)$ where $M$ is a $\Cattakeshi$-module and 
\[
m\in\coker\left(\bigoplus_{x,x'\in X,\: y\in\inf(x,x')}M(\gosi y)\xrightarrow{\begin{pmatrix} i_{\gosi y}^{\gosi x} & -i_{\gosi y}^{\gosi x'} \end{pmatrix}} \bigoplus_{x\in X} M(\gosi x)\right),
\]
 and a morphism $\phi\colon(M,m)\to(N,n)$ is a $\Cattakeshi$-morphism whose induced map on the cokernels sends $m$ to $n$.
 
Similarly, the categories $\Mod(\Catgunnartakeshi)^\pt$  and $\Mod(\Catgunnar)^\pt$ of \emph{pointed $\Catgunnartakeshi$-modules} respectively \emph{pointed $\Catgunnar$-modules} are defined.
\end{definition}

\begin{definition}
A pointed $\Catsixterm$-module $(M,m)$ is called \emph{exact} if $M$ is an exact $\Catsixterm$-module, and \emph{\rrzero{}} if $M$ is \rrzero .
Similary, a pointed $\Cattakeshi$-module, $\Catgunnartakeshi$-module, or $\Catgunnar$-module $(M,m)$ is called \emph{exact} if $M$ is exact.
\end{definition}

\begin{lemma} \label{lem:seq_unit}
Let $M$ be an exact \rrzero{} $\Catsixterm$-module.  Then the sequence
\[
\bigoplus_{x,x'\in X,\: y\in\inf(x,x')} M(\osi y,0) \xrightarrow{\begin{pmatrix} i_{\osi y}^{\osi x} & -i_{\osi y}^{\osi x'} \end{pmatrix}}
\bigoplus_{x\in X} M(\osi x,0) \xrightarrow{(i_{\osi x}^X)} M(X,0) \to 0
\]
is exact.
\end{lemma}
\begin{proof}
By Lemma~\ref{lem:exactseq} the horizontal row of the following commuting diagram is exact:
\[
\xymatrix@C-1pt@R+14pt{
\displaystyle \bigoplus_{x,x'\in X} M(\osi x\cap \osi{x'},0) \ar[rrr]^-{\scriptscriptstyle \begin{pmatrix} i_{\osi x\cap\osi{x'}}^{\osi x} & -i_{\osi x\cap\osi{x'}}^{\osi x'} \end{pmatrix}} &&& \displaystyle\bigoplus_{x\in X} M(\osi x,0) \ar[r]^-{} & M(X,0) \ar[r] & 0 \\
\displaystyle\bigoplus_{\stackrel{x,x'\in X}{y\in\inf(x,x')}}M(\osi y,0) \ar[u]^-{(i_{\osi y}^{\osi x\cap\osi{x'}})}  \ar[urrr]_(.5){\qquad \scriptscriptstyle \begin{pmatrix} i_{\osi y}^{\osi x} & -i_{\osi y}^{\osi x'} \end{pmatrix}}  &&&&&
}
\]
Furthermore, since for any pair $x,x'\in X$ the collection $(\osi y)_{y\in\inf(x,x')}$ covers $\osi x \cap \osi{x'}$, we see by Lemma~\ref{lem:exactseq} that the vertical map in the diagram is surjective.  This establishes the desired result.
\end{proof}

\begin{definition}
Let $A$ be a unital \csa{} over a finite $T_0$-space $X$.
Its \emph{unital concrete filtered $\K$\nb-the\-ory} $\FKsixterm^\unit(A)$ is defined as the pointed $\Catsixterm$-module $(\FKsixterm(A),[1_A])$.

If $A$ has real rank zero, then its \emph{unital reduced filtered $\K$\nb-the\-ory} $\FKgunnar^\unit(A)$ is defined as the pointed $\Catgunnar$-module $(\FKgunnar(A),u(A))$ where $u(A)$ is the unique element in
\[
\coker\left( \bigoplus_{x,x'\in X,\: y\in\inf(x,x')} \FK_{\osi y}^0(A)\xrightarrow{\begin{pmatrix} i_{\gosi y}^{\gosi x} & -i_{\gosi y}^{\gosi x'} \end{pmatrix}}\bigoplus_{x\in X} \FK_{\osi x}^0(A)\right)
\]
that is mapped to $[1_A]$ in $\K_0(A)$ by the map induced by the family $\bigl(\FK_{\osi x}^0(A)\xrightarrow{i_{\osi x}^X}\FK_X^0(A)\bigr)_{x\in X}$, see~Lemma~\ref{lem:seq_unit}.

If $A$ has real rank zero and $X$ is a unique path space, then its \emph{unital filtered $\K$\nb-the\-ory restricted to the canonical base} $\FKtakeshi^\unit(A)$ is defined similarly.
\end{definition}

By Lemma~\ref{lem:seq_unit}, we may view the forgetful functor $\Ftakeshi\colon\Mod(\Catsixterm)\to\Mod(\Cattakeshi)$ as a functor from \emph{pointed} exact \rrzero{} $\Catsixterm$-modules to \emph{pointed} exact $\Cattakeshi$-modules and immediately obtain the following pointed version of Theorem~\ref{thm:sixterm_to_takeshi}:

\begin{proposition}
  \label{pro:pointed_sixterm_to_takeshi}
For every EBP space $X$, the forgetful functor from exact pointed \rrzero{} $\Catsixterm$-modules to exact pointed $\Cattakeshi$-modules is an equivalence of categories.
\end{proposition}

\begin{proposition}
Assume that $X$ is a unique path space.  Let $(M,m)$ and $(N,n)$ be exact pointed $\Catgunnartakeshi$-modules with $M(\gsi x)$ and $N(\gsi x)$ free for all non-open points $x\in X$, and let $\phi\colon\FGTgunnar(M)\to\FGTgunnar(N)$ be a pointed $\Catgunnar$-module homomorphism. Then there exists a \textup{(}not necessarily unique\textup{)} pointed $\Catgunnartakeshi$-module homomorphism $\Phi\colon M\to N$ satisfying $\FGTgunnar(\Phi)=\phi$, and if $\phi$ is an isomorphism, then $\Phi$ is by construction an isomorphism.
\end{proposition}

\begin{proof}
This follows from Theorem~\ref{thm:gunnartotakeshi} since the groups $M(\gosi x)$ are not forgotten by $\FGTgunnar$.
\end{proof}

\begin{corollary} \label{cor:unit_takeshi}
Let $X$ be an accordion space, and let $A$ and $B$ be unital Kirchberg $X$\nb-al\-ge\-bras of real rank zero with all simple subquotients in the bootstrap class.
Then any isomorphism $\FKtakeshi^\unit(A)\to\FKtakeshi^\unit(B)$ lifts to an $X$\nb-equi\-vari\-ant \Star{}isomorphism $A\to B$.
\end{corollary}

\begin{proof}
This follows from Theorem~3.3 in~\cite{err_meta} together with Theorem~\ref{bentmann} and Corollary~\ref{cor:sixterm_to_takeshi}.
\end{proof}

\begin{corollary} \label{cor:unit_gunnar}
Let $X$ be an accordion space, and let $A$ and $B$ be unital Kirchberg $X$\nb-al\-ge\-bras of real rank zero with all simple subquotients in the bootstrap class.
Assume that $\K_1\bigl(A(x)\bigr)$ and $\K_1\bigl(B(x)\bigr)$ are free abelian groups for all $x\in X$.
Then any isomorphism $\FKgunnar^\unit(A)\to\FKgunnar^\unit(B)$ lifts to an $X$\nb-equi\-vari\-ant \Star{}isomorphism $A\to B$.
\end{corollary}

\begin{proof}
This follows from Theorem~3.3 in~\cite{err_meta} together with Theorem~\ref{bentmann} and Corollary~\ref{cor:gunnartosixterm}.
\end{proof}

\begin{remark} \label{rem:unital}
There exist, up to homeomorphism, precisely four contractible unique path spaces with four points that are not accordion spaces. For all these spaces, the categories $\NT$ and $\Catsixterm$ coincide. In~\cite{arr}, Gunnar Restorff, Efren Ruiz and the first-named author showed that if $X$ is one of these spaces, then $\FK$ is a complete invariant for stable Kirchberg $X$\nb-al\-ge\-bras of real rank zero. Therefore, Corollaries~\ref{cor:unit_takeshi} and~\ref{cor:unit_gunnar} also hold for these spaces. Furthermore, the proof of Proposition~\ref{prop:diamond} also applies to $\FKgunnar^\unit$ and unital \csa s, hence Corollary~\ref{cor:unit_gunnar} also holds for the space $\mathcal D$.
\end{remark}

We now recall the unital version of the range result from~\cite{range_result}.

\begin{theorem}[\cite{range_result}*{Theorem~5.5}] \label{thm:range_unit}
Let $X$ be a finite $T_0$-space, and let $(M,m)$ be an exact pointed $\Catgunnar$-module.  Assume that for all $x\in X$, $M(\gsi x)$ is a free abelian group,
\[ \coker \bigl(M(\gobd x)\xrightarrow{i_{\gobd x}^{\gosi x}} M(\gosi x)\bigr) \]
 is finitely generated, and $\rank M(\gsi x)\leq \rank \coker \bigl(M(\gobd x)\xrightarrow{i_{\gobd x}^{\gosi x}} M(\gosi x)\bigr)$.

Then there exists a countable graph $E$ satisfying that all vertices in $E$ support at least two cycles, that $E^0$ is finite, that $C^*(E)$ is tight over~$X$, and that $\FKgunnar^\unit\bigl(C^*(E)\bigr)$ is isomorphic to $(M,m)$. By construction $C^*(E)$ is unital and purely infinite.

The graph $E$ can be chosen to have only regular vertices if \textup{(}and only if\textup{)} the rank of $M(\gsi x)$ coincides with the rank of the cokernel of $i\colon M(\gobd x)\to M(\gosi x)$ for all $x\in X$.  
If $E$ is chosen to have only regular vertices, then by construction $C^*(E)$ is a Cuntz--Krie\-ger algebra.
\end{theorem}

\begin{corollary}
Let $X$ be a finite $T_0$-space and let $A$ be a unital \csa{} over~$X$ of real rank zero. Assume for all $x\in X$ that $\K_1\bigl(A(x)\bigr)$ is free, $\K_0\bigl(A(x)\bigr)$ is finitely generated, and $\rank\K_1\bigl(A(x)\bigr)\leq\rank\K_0\bigl(A(x)\bigr)$.

Then there exists a countable graph $E$ for which $C^*(E)$ is unital, purely infinite, and tight over~$X$ such that $\FKgunnar^\unit\bigl(C^*(E)\bigr)\cong\FKgunnar^\unit(A)$.
If $X$ is an EBP space, then automatically $\FKsixterm^\unit\bigl(C^*(E)\bigr)\cong\FKsixterm^\unit(A)$.

If furthermore $\rank\K_1\bigl(A(x)\bigr)=\rank\K_0\bigl(A(x)\bigr)$ for all $x\in X$, then $E$ can be chosen such that $C^*(E)$ is a purely infinite Cuntz--Krie\-ger algebra.
\end{corollary}

\begin{corollary} \label{cor:extensions}
Let $X$ be an accordion space, and let $I\into A\onto B$ be an extension of \csa s.  Assume that $A$ is unital and tight over~$X$.

Then $A$ is a purely infinite Cuntz--Krie\-ger algebra if and only if
\begin{itemize}
\item $I$ is stably isomorphic to a purely infinite Cuntz--Krie\-ger algebra,
\item $B$ is a purely infinite Cuntz--Krie\-ger algebra,
\item the exponential map $\K_0(B)\to\K_1(I)$ vanishes.
\end{itemize}

\end{corollary}
\begin{proof}
Recall that \CK s are purely infinite if and only if they have real rank zero.
Assume that $A$ is a purely infinite \CK{}. It is well-known that then $B$ is also a purely infinite \CK{} and $I$ is stably isomorphic to one.
By Theorem~4.2 of~\cite{pasnicurordam}, $\K_0(B)\to\K_1(I)$ vanishes since $A$ has real rank zero and therefore is $\K_0$-liftable, see~Remark~\ref{rem:K0lift}.

Now, assume that $B$ is a purely infinite \CK{}, that $I$ is stably isomorphic to one, and that the map $\K_0(B)\to\K_1(I)$ vanishes.
By Theorem~4.3 of~\cite{tomswinter}, $A$ is $\mathcal O_\infty$-absorbing since $B$ and $I$ are.
Since $B$ and $I$ are $\K_0$-liftable and $\K_0(B)\to\K_1(I)$ vanishes, $A$ is also $\K_0$-liftable (that is, $\FK(A)$ is \rrzero{}) by \cite{Bentmann:Intermediate_cancellation}*{Proposition~3.5}. So by pure infiniteness of $A$ it therefore follows from Theorem~4.2 of~\cite{pasnicurordam} that $A$ has real rank zero.
For all $x\in X$, $\K_1\bigl(A(x)\bigr)$ is free since $B$ and $I$ are stably isomorphic to \CK s. So by Theorem~\ref{thm:range_unit} there exists a real-rank-zero \CK{} $C$ that is tight over~$X$ and has $\FKgunnar^\unit(A)\cong\FKgunnar^\unit(C)$.  By Corollary~\ref{cor:unit_gunnar}, $A$ and $C$ are isomorphic.
\end{proof}

\begin{remark}
Corollary~\ref{cor:extensions} holds in fact for all spaces~$X$ for which $\FKgunnar^\unit$ is a complete invariant for unital Kirchberg $X$\nb-al\-ge\-bras $A$ where $A(x)$ is in the bootstrap class and $\K_1\bigl(A(x)\bigr)$ is free for all $x\in X$, see~Remark~\ref{rem:unital}.
\end{remark}

\section{Ordered filtered \texorpdfstring{$\K$}{K}-theory} \label{sec:order}
The notion of ordered filtered $\K$\nb-the\-ory was introduced by S\o ren Eilers, Gunnar Restorff, and Efren Ruiz in \cite{err} to classify certain (not necessarily purely infinite) graph \csa{}s of real rank zero. We hope that the results in this section will be useful for future work in this direction.

Recall that for a \csa{} $A$, a class in $\K_0(A)$ of the from $[p]_0$ for a projection~$p$ in $M_n(A)$ for some $n\in\N$ is called \emph{positive}.
The \emph{positive cone} $\K_0(A)^+$ consists of all positive elements in $\K_0(A)$.
For two \csa s $A$ and $B$, a group homomorphism $\phi\colon\K_0(A)\to\K_0(B)$ is called \emph{positive} if $\phi(\K_0(A)^+)\subseteq\K_0(B)^+$, and a group isomorphism $\phi\colon\K_0(A)\to\K_0(B)$ is called an \emph{order isomorphism} if $\phi(\K_0(A)^+)=\K_0(B)^+$.

Note that for a finite topological space $X$,  a locally closed subset $Y$ of $X$, and an open subset $U$ of $Y$, the maps $i_U^Y\colon\K_0\bigl(A(U)\bigr)\to\K_0\bigl(A(Y)\bigr)$ and $r_Y^{Y\setminus U}\colon\K_0\bigl(A(Y)\bigr)\to\K_0\bigl(A(Y\setminus U)\bigr)$ are positive.

\begin{definition}
For \csa s $A$ and $B$ over a finite topological space $X$, an $\Catsixterm$-module homomorphism $\phi\colon\FKsixterm(A)\to\FKsixterm(B)$ is called \emph{positive} if the induced maps $\FK_Y^0(A)\to\FK_Y^0(B)$ are positive for all $Y\in\LC(X)$, and an $\Catsixterm$-module isomorphism $\FKsixterm(A)\to\FKsixterm(B)$ is called  an \emph{order isomorphism} if the induced isomomorphisms are order isomorphisms.
For the reduced versions $\FKgunnar$, $\FKtakeshi$, and $\FKgunnartakeshi$ of filtered $\K$\nb-the\-ory, analogous definitions apply.
\end{definition}

We are indebted to Mikael R\o rdam for the elegant proof of the following lemma.

\begin{lemma} \label{lemma:positivecone}
Let $A$ be a real-rank-zero \csa{} and let $I$ and $J$ be \textup{(}closed, two-sided\textup{)} ideals in $A$ satisfying $I+J=A$. Then any projection $p$ in $A$ can be written as $p=q+q'$ with a projection~$q$ in~$I$ and a projection~$q'$ in~$J$.
\end{lemma}
\begin{proof}
Let $p$ a projection in $A$ be given and write $p=a+b$ with $a\in I$ and $b\in J$.  We may assume that $a=pap$ and $b=pbp$.
As $A$ has real rank zero, the hereditary subalgebra $pIp$ has an approximate unit of projections, so there exists a projection $q$ in $pIp$ satisfying $\|a-aq\|<1$. Since $q=pqp$, $q\leq p$ and we may define a projection $q'$ as $q'=p-q$.
It remains to prove $q' \in J$. 
We have 
\[
\|q'-q'bq'\| 
= \|q'(p-b)q'\| = \|q'a(p-q)\| \leq \|q'\| \|a-aq\|<1. 
\]
Since $q'bq'\in J$, 
the image of $q'$ in the quotient $A/J$ is 
a projection of norm strictly less than $1$. 
Since such a projection is $0$, we get $q' \in J$. 
\end{proof}

The following theorem is a version of Corollary~\ref{cor:sixterm_to_takeshi} taking the order into account.

\begin{theorem}
Let $X$ be an EBP space, and let $A$ and $B$ be \csa s over~$X$ of real rank zero.
Then for any order isomorphism $\phi\colon\FKtakeshi(A)\to\FKtakeshi(B)$  there is a unique order isomorphism $\Phi\colon\FKsixterm(A)\to\FKsixterm(B)$ satisfying $\Ftakeshi(\Phi)=\phi$.
\end{theorem}

\begin{proof}
By Corollary~\ref{cor:sixterm_to_takeshi}, $\Phi$ is an isomorphism if and only if $\phi$ is.  Assume that $\phi$ is an order isomorphism, and let us show first for $Y\in\Op(X)$ and then for $Y\in\LC(X)$ that $\Phi_Y^0$ is an order isomorphism.

For $U$ an open subset of $X$, the following diagram has commuting squares and its rows are exact by Lemmas~\ref{lem:cosheaf_equivalence} and~\ref{lem:presheaf_on_canonical_base_is_sheaf}.
\[ \xymatrix@C+25pt{
\displaystyle\bigoplus_{y\in\inf(x,x')} \FK_{\widetilde{\{y\}}}^0(A) \ar[r]^-{{\begin{pmatrix} i_{\osi y}^{\osi x} & -i_{\osi y}^{\osi x'} \end{pmatrix}}}\ar[d]^-{(\phi_{\widetilde{\{y\}}}^0)} & \displaystyle\bigoplus_{x\in U} \FK_{\widetilde{\{x\}}}^0(A) \ar[r]^-{(i_{\widetilde{\{x\}}}^U)} \ar[d]^-{(\phi_{\widetilde{\{x\}}}^0)} & \FK_U^0(A) \ar[r] \ar[d]^-{\Phi_U^0} & 0 \\
\displaystyle\bigoplus_{y\in\inf(x,x')} \FK_{\widetilde{\{y\}}}^0(B) \ar[r]^-{{\begin{pmatrix} i_{\osi y}^{\osi x} & -i_{\osi y}^{\osi x'} \end{pmatrix}}} & \displaystyle\bigoplus_{x\in U} \FK_{\widetilde{\{x\}}}^0(B) \ar[r]^-{(i_{\widetilde{\{x\}}}^U)} & \FK_U^0(B) \ar[r] & 0
} \]
Since $(A\bigl(\widetilde{\{x\}})\bigr)_{x\in U}$ is a finite collection of ideals in $A(U)$, we see by Lemma~\ref{lemma:positivecone} that the map $(i_{\widetilde{\{x\}}}^U)\colon\displaystyle\bigoplus_{x\in U}\K_0(A({\widetilde{\{x\}}})\to\K_0\bigl(A(U)\bigr)$ surjects $\displaystyle\bigoplus_{x\in U}\K_0(A({\widetilde{\{x\}}})^+$ onto $\K_0\bigl(A(U)\bigr)^+$.
Similarly, the map $(i_{\widetilde{\{x\}}}^U)\colon\displaystyle\bigoplus_{x\in U}\K_0(B({\widetilde{\{x\}}})\to\K_0\bigl(B(U)\bigr)$ surjects $\displaystyle\bigoplus_{x\in U}\K_0(B({\widetilde{\{x\}}})^+$ onto $\K_0\bigl(B(U)\bigr)^+$.
A simple diagram chase therefore shows that $\Phi_U^0$ is an order isomorphism since the map $\phi_{\widetilde{\{x\}}}^0$ is an order isomorphism for all $x\in U$.

For a locally closed subset~$Y$ of~$X$, choose open subsets~$U$ and~$V$ of~$X$ satisfying $V\subseteq U$ and $U\setminus V=Y$.  Then $\Phi_U^0$ is an order isomorphism.  Consider the following diagram with exact rows and commuting squares.
\[ \xymatrix{
\FK_V^0(A) \ar[r]^-{i_V^U} \ar[d]^-{\Phi_V^0} & \FK_U^0(A) \ar[r]^-{r_U^Y} \ar[d]^-{\Phi_U^0} & \FK_Y^0(A) \ar[r] \ar[d]^-{\Phi_Y^0} & 0 \\
\FK_V^0(B) \ar[r]^-{i_V^U} & \FK_U^0(B) \ar[r]^-{r_U^Y} & \FK_Y^0(B) \ar[r] & 0
} \]
In~\cite{brownpedersen}*{Theorem 3.14}, Lawrence G.~Brown and Gert K.~Pedersen showed that given an extension $I\into C\onto C/I$ of \csa s, the \csa{} $C$ has real rank zero if and only if  $I$ and $C/I$ have real rank zero and projections in $C/I$ lift to projections in $C$.
Thus, since~$A$ and therefore $M_n\otimes A(U)$ for all~$n$ has real rank zero, the map $i_U^Y\colon\K_0\bigl(A(U)\bigr)\to\K_0\bigl(A(Y)\bigr)$ surjects $\K_0\bigl(A(U)\bigr)^+$ onto $\K_0\bigl(A(Y)\bigr)^+$.
Similarly, the map $i_U^Y\colon\K_0\bigl(B(U)\bigr)\to\K_0\bigl(B(Y)\bigr)$ surjects $\K_0\bigl(B(U)\bigr)^+$ onto $\K_0\bigl(B(Y)\bigr)^+$.  
A simple diagram chase therefore shows that $\Phi_Y^0$ is an order isomorphism.
\end{proof}

We have the following ordered analogs of Theorem~\ref{thm:gunnartotakeshi} and Corollary~\ref{cor:gunnartosixterm_cstar}.

\begin{theorem}
Let $X$ be a unique path space, and let $A$ and $B$ be \csa s over~$X$ of real rank zero.  Assume that $\K_1\bigl(A(\{x\})\bigr)$ and $\K_1\bigl(B(\{x\})\bigr)$ are free abelian groups for all non-open points $x\in X$.
Then for any order isomorphism $\phi\colon\FKgunnar(A)\to\FKgunnar(B)$ there exists a \textup{(}not necessarily unique\textup{)} order isomorphism $\Phi\colon\FKgunnartakeshi(A)\to\FKgunnartakeshi(B)$ that satisfies $\FGTgunnar(\Phi)=\phi$.
\end{theorem}

\begin{proof}
Since the functor $\FGTgunnar$ only forgets $\K_1$-groups, the desired follows immediately from Theorem~\ref{thm:gunnartotakeshi}.
\end{proof}

\begin{corollary}
Let $X$ be an EBP space, and let $A$ and $B$ be \csa s over~$X$ of real rank zero.  Assume that $\K_1\bigl(A(\{x\})\bigr)$ and $\K_1\bigl(B(\{x\})\bigr)$ are free abelian groups for all $x\in X$.
Then for any order isomorphism $\phi\colon\FKgunnar(A)\to\FKgunnar(B)$ there exists a order isomorphism $\Phi\colon\FKsixterm(A)\to\FKsixterm(B)$ that satisfies $\Fgunnar(\Phi)=\phi$.
\end{corollary}

\begin{proof}
Combine the previous two theorems.
\end{proof}

\section{Corollaries for accordion spaces}
  \label{sec:accordion}
  
We summarize our results in the most satisfying case of accordion spaces. By combining Theorems~\ref{bentmann}, \ref{thm:sixterm_to_takeshi}, \ref{thm:range} and Corollaries~\ref{cor:STexactR}, \ref{cor:gunnartosixterm} in the stable case and Proposition~\ref{pro:pointed_sixterm_to_takeshi}, Corollaries~\ref{cor:STexactR}, \ref{cor:unit_takeshi}, \ref{cor:unit_gunnar}, and Theorem \ref{thm:range_unit} in the unital case, we obtain the following characterization of purely infinite graph \csa{}s, and of purely infinite Cuntz--Krie\-ger algebras. In the first list, we use that the stabilization of a graph \csa{} is again a graph \csa{} by \cite{gamt:isomorita}*{Proposition~9.8(3)}.

\begin{corollary} \label{cor:main}
Let $X$ be an accordion space. The different versions of filtered $\K$\nb-the\-ory introduced in this article induce bijections between the sets of isomorphism classes of objects in the following three lists, respectively.

 \textbf{\textup{List 1:}}
\begin{itemize}
\item
tight, stable, purely infinite graph \csa{}s over~$X$,
\item
stable Kirchberg $X$\nb-al\-ge\-bras $A$ of real rank zero with all simple subquotients in the bootstrap class satisfying that $\K_1\bigl(A(\{x\})\bigr)$ is free for all $x\in X$,
\item
countable, exact, \rrzero{} $\NT$-modules $M$ with $M(\{x\},1)$ free for all $x\in X$,
\item
countable, exact $\Cattakeshi$-modules $M$ with $M(\gsi x)$ free for all $x\in X$,
\item
countable, exact $\Catgunnar$-modules $M$ with $M(\tcsi x)$ free for all $x\in X$.
\end{itemize}

 \textbf{\textup{List 2:}}
\begin{itemize}
\item
tight, unital, purely infinite graph \csa{}s over~$X$,
\item
unital Kirchberg $X$\nb-al\-ge\-bras $A$ of real rank zero, with all simple subquotients in the bootstrap class such that, for all $x\in X$, the group $\K_1\big(A(\{x\})\bigr)$ is free and $$\rank\K_1\bigl(A(\{x\})\bigr)\leq\rank\K_0\bigl(A(\{x\})\bigr)<\infty,$$
 \item
countable, exact, \rrzero{} pointed $\NT$-modules $M$ such that, for all $x\in X$, the group $M(\{x\},1)$ is free and $$\rank\bigl(M(\si x,1)\bigr)\leq\rank\bigl(M(\si x,0)\bigr)<\infty,$$
 \item
countable, exact pointed $\Cattakeshi$-modules $M$ such that, for all $x\in X$, the group $M(\gsi x)$ is free and $$\rank\bigl(M(\gsi x)\bigr)\leq\rank\Bigl(\coker\bigl(\bigoplus_{y\to x} M(\gosi y) \to M(\gosi x)\bigr)\Bigr)<\infty,$$
\item
isomorphism classes of countable, exact pointed $\Catgunnar$-modules $M$ such that, for all $x\in X$, the group $M(\tcsi x)$ is free and $$\rank\bigl(M(\gsi x)\bigr)\leq\rank\Bigl(\coker\bigl(M(\gobd x)\to M(\gosi x)\bigr)\Bigr)<\infty.$$
\end{itemize}

 \textbf{\textup{List 3:}}
 \begin{itemize}
 \item
tight, purely infinite Cuntz--Krie\-ger algebras over~$X$,
\item
unital Kirchberg $X$\nb-al\-ge\-bras $A$ of real rank zero, with all simple subquotients in the bootstrap class such that, for all $x\in X$, the group $\K_1\big(A(\{x\})\bigr)$ is free and $$\rank\K_1\bigl(A(\{x\})\bigr)=\rank\K_0\bigl(A(\{x\})\bigr)<\infty,$$
 \item
countable, exact, \rrzero{} pointed $\NT$-modules $M$ such that, for all $x\in X$, the group $M(\{x\},1)$ is free and $$\rank\bigl(M(\si x,1)\bigr)=\rank\bigl(M(\si x,0)\bigr)<\infty,$$
 \item
countable, exact pointed $\Cattakeshi$-modules $M$ such that, for all $x\in X$, the group $M(\gsi x)$ is free and $$\rank\bigl(M(\gsi x)\bigr)=\rank\Bigl(\coker\bigl(\bigoplus_{y\to x} M(\gosi y) \to M(\gosi x)\bigr)\Bigr)<\infty,$$
\item
countable, exact pointed $\Catgunnar$-modules $M$ such that, for all $x\in X$, the group $M(\tcsi x)$ is free and $$\rank\bigl(M(\gsi x)\bigr)=\rank\Bigl(\coker\bigl(M(\gobd x)\to M(\gosi x)\bigr)\Bigr)<\infty.$$
 \end{itemize}
 \end{corollary}

\end{document}